\RequirePackage{fix-cm}
\documentclass[smallcondensed]{svjour3}     
\smartqed  
\usepackage{graphicx}
\usepackage{amssymb}
\newcommand{\R}{{\Bbb R}}

\usepackage{amsmath}

\begin{document}

\title{Speed selection and stability of wavefronts for delayed monostable reaction-diffusion equations
}

\titlerunning{Speed selection and stability of wavefronts for delayed monostable equations}        

\author{Abraham Solar         \and
        Sergei Trofimchuk 
}


\institute{A. Solar  \at
              Instituto de Matem\'atica y F\'isica, Universidad de Talca, Casilla 747,
Talca, Chile \\
              \email{asolar.solar@gmail.com}           
           \and
           S. Trofimchuk (corresponding author) \at Instituto de Matem\'atica y F\'isica, Universidad de Talca, Casilla 747,
Talca, Chile \\
              \email{ trofimch@inst-mat.utalca.cl}  }


\maketitle

\begin{abstract} We study the asymptotic stability of traveling fronts and front's velocity selection problem for the time-delayed
monostable  equation $(*)$ $u_{t}(t,x) = u_{xx}(t,x) - u(t,x) + g(u(t-h,x)),$ $x \in \R,\ t >0,$ considered with Lipschitz continuous reaction term $g: \R_+ \to \R_+$. We are also  assuming that  
 $g$ is $C^{1,\alpha}$-smooth in some neighbourhood of the equilibria $0$ and $\kappa >0$ to  $(*)$. 
In difference with the previous works, we do not impose any convexity or subtangency condition on the graph of $g$ so that equation $(*)$ can possess  pushed traveling fronts.   Our first main result says that  the non-critical wavefronts of $(*)$ with monotone $g$ are globally  
nonlinearly  stable.  In the special and easier  case when  the Lipschitz constant  for $g$ coincides with  $g'(0)$, we present a series of results concerning the exponential [asymptotic] stability   of non-critical [respectively, critical] fronts   for  monostable model $(*)$. As an application, we present a criterion of the absolute global stability of non-critical wavefronts to the diffusive Nicholson's blowflies equation.  
\end{abstract}



\newpage

\section{Introduction and main results} \label{intro}


Set $\Pi_0:=[-h,0]\times \R \subset \R^2$ and consider the family $\mathcal{F}$ of continuous and uniformly bounded functions $w_0(s,x)$, $w_0:\Pi_0\to \R_+$, exponentially decaying (uniformly in $s$) as $x\to -\infty$ and separated from $0$  (uniformly in $s$) as $x\to +\infty$. In particular, we assume that each $w_0\in \mathcal{F}$ satisfies
\begin{description}
 \item [$(IC1)$] $\quad 0\leq w_0(s,x)\leq |w_0|_\infty:= \sup_{(s,x)\in \Pi_0}w_0(s,x) < \infty, \quad  (s,x) \in \Pi_0;$ 
 
 \vspace{2mm}
 
 \item[$(IC2)$] \quad
$\liminf_{x \to +\infty} \min_{s\in[-h,0]} w_0(s,x)> 0$.\end{description}

Everywhere in the sequel, we  will also assume that each element $w_0(s,x)$  of $\mathcal{F}$ is locally H\"older continuous in $x \in \R$,
uniformly with respect to $s\in [-h,0]$. 

Our goal in this work is to indicate subclasses of   initial functions  $w_0\in  \mathcal{F}$ 
for monostable reaction-diffusion equations with monotone 
delayed reaction
\begin{eqnarray} \label{e1}
u_{t}(t,x) &=& u_{xx}(t,x) - u(t,x) + g(u(t-h,x)), \  t>0,\ x\in\R, 
 \\
u(s,x) & =& w_0(s,x), \ s \in [-h,0], \ x \in \R, \label{e2}
\end{eqnarray}
which yield solutions $u= u(t,x,w_0)$  converging, as $t \to
+\infty$, to appropriate  traveling fronts  $u = \phi(x+ct, w_0)$, $c>0$, of (\ref{e1}), (\ref{e2}).   By definition, 
the front profile $\phi:\R \to \R_+$ is a positive bounded smooth function such that the limits $\phi(-\infty)=0$, $\phi(+\infty)=\kappa$ exist. 
Here we are assuming that the continuous nonlinearity $g:
\mathbb R_+ \to \mathbb R_+$ satisfies the  monostability condition
\vspace{1mm}

\noindent  {\rm \bf(H)}  
the equation $g(x)= x$ has exactly two nonnegative solutions: $0$ and
$\kappa >0$. Moreover, $g$  is $C^1$-smooth in some
$\delta_0$-neighborhood of the equilibria  where $g'(0) >1,$
$g'(\kappa) < 1,$ and it also satisfies the Lipshitz condition
$|g(u)-g(v)| \leq L_g|u-v|,$ $u,v \geq 0$. In addition,
there are $C >0,\ \theta \in (0,1],$ such that   $
\left|g'(u)- g'(0)\right| \leq Cu^\theta $ for $u\in
(0,\delta_0].$ To simplify the notation, we will  extend  $g$ linearly and $C^1-$smoothly on $(-\infty,0]$.

\vspace{1mm}

From \cite{AF} we know that above conditions imposed on $w_0$  are sufficient for the existence of a unique classical solution
$u= u(t,x,w_0): [-h,+\infty) \times \R \to \R_+$ to (\ref{e1}),
(\ref{e2}) (i.e. of a continuous and bounded (at least, on finite time intervals) function $u$ having continuous
derivatives $u_t, u_x, u_{xx}$ in $\Omega = (0,+\infty)\times \R$
and satisfying (\ref{e1}) in $\Omega$ as well as (\ref{e2}) in
$[-h,0]\times \R$). 
 We will show that, similarly to $w_0(s,x)$ and $\phi(x+cs,w_0)$, the function $w_{(t)}(s,x)=u(t+s,x,w_0)$, $(s,x)\in \Pi_0,$ will also belong to the class $\mathcal{F}$, for each fixed $t>0$.
 
In this way, the concept of `speed selection' reflects the evident fact that the properties of $w_0$ may 
determine the speed of propagation of the initial `concentration' (of something) $w_0(s,x)$ from the right side of
the $x$-axis $\R$ (where $w_0$ is separated from $0$) to the left side of $\R$ (where $w_0$ vanishes). Moreover,   
in the non-delayed case (when $h =0$) it is well known  \cite{ES} that, given a converging  solution $u(t,x,w_0) \leadsto  \phi(x+ct,w_0)$, the speed of propagation $c$ `choosen' by $u(t,x,w_0)$ depends mainly only on the  asymptotic behavior of $w_0(s,x)$ at $x =-\infty$.  It is clear also  that the speed selection problem is closely related to the front stability question: indeed, if some wavefront  $u = \phi(x+c_0t)$ is stable (in an appropriate metric phase space), then  each initial datum $w_0(s,x)$ close  to $\phi(x+c_0s)$ yields a  `concentration' distribution $u(t,x)$ propagating to the left of $\R$ with the same velocity $c_0$.  Below we will give precise 
mathematical formulations for the above  informal discussion. 

The studies of wavefront stability in monostable monotone delayed model (\ref{e1}) (including its non-local and discrete Laplacian versions) were initiated in 2004-2005 by Mei {\it et al.} \cite{MeiF} and  Ma and Zou \cite{MaZou}. 
Their research was influenced by a series of previous results about a) the existence of monotone wavefronts \cite{ma,wz}; b) the stability of wavefronts in delayed  bistable  equations \cite{OM,SZ} and discrete monostable 
equations \cite{CG1}.  Over the last decade, the wave stability problem for equation (\ref{e1}) has attracted attention of many other mathematicians so that it would be difficult to mention all interesting findings in this area. We believe, however, that the strongest results concerning the wavefront stability in the {\it monotone} Mackey-Glass type reaction-diffusion equation (\ref{e1}) can be found in \cite{LvW,MLLS,MLLS2,MeiI,WLR} (see also \cite{CMYZ,IGT,FGT,KGB,TT,WZL} and references therein for the case of  unimodal birth function $g$).  
In our work,  rather then writing statements of the aforementioned  results from \cite{LvW,MLLS,MLLS2,MeiI,WLR}, we prefer to  discuss their relations with  our two main theorems announced below.  

Now, two different approaches were employed in 
the cited works: a weighted energy approach \cite{MeiF,MLLS,MLLS2,MeiI} and the super- and sub-solution method \cite{MaZou,WLR}.  
The stability of monotone wavefronts to (\ref{e1}) was always proved under rather strong smoothness  ($C^2$-smoothness) and shape conditions on $g$. In particular, hypotheses imposed on $g$  were always  sufficient to assure the inequality  $g'(x)\leq g'(0)$ for all $x\in [0,\kappa]$ (cf. \cite[Subsection 1.2]{TPT}).  The latter condition, however, excludes a subclass of equations (\ref{e1}) possessing so called pushed minimal traveling fronts \cite{ES,TPT}. Since pushed wavefronts are quite interesting from both applied \cite{JP,RGHK} and mathematical  \cite{BGHR,ID,HR,LZh,roth,ES,ST,TPT} points of view, their existence, uniqueness  and stability properties in the case of delayed monotone model (\ref{e1}) were recently considered in \cite{LZh,STR,TPT}.  Particularly, the existence of the minimal 
speed of front propagation $c_*$  was proved in \cite{LZh,TPT} (if $g$ is neither monotone nor subtangential at $0$, the existence of $c_*$ is an important open problem). It should be also observed that, in general, either analytical determination or numerical  approximation of the exact value of $c_*$ is a quite difficult task \cite{BD,HR,ES,TPT}. By \cite{LZh,STR}, $c_*$ coincides with the {\it asymptotic speed of propagation} (this important concept was proposed by Aronson and Weinberger  \cite{AW} in 1977). Next,  the stability of pushed wavefronts to (\ref{e1}) was also investigated  in \cite{STR}. 

In the present work, we continue our  studies  in  \cite{STR}, by analysing stability of other (i.e. not necessarily minimal) wavefronts $u =\phi(x+ct),$ $c \geq c_*$, to equation (\ref{e1}). One of the main difference with the previous works consists in  generally non-convex and non-smooth nature of the monotone  birth function $g$: for instance, in our first results below,  we do not even require  the subtangency condition $$g(x)\leq g'(0)x, \ x \geq 0.$$ 

Before announcing our first theorem, we recall \cite{TPT} that the condition  $c \geq c_*$ implies that the characteristic equation at the trivial steady state 
$$
\chi_0(\lambda):= \lambda^2-c\lambda-1 + g'(0)e^{-\lambda ch}=0
$$
has exactly two real roots $\lambda_1=\lambda_1(c) \leq \lambda_2=\lambda_2(c)$ (counting multiplicity), both of them are positive.   Note also that $-\lambda_1(c), \lambda_2(c)$ are increasing functions of $c$.

Next, for a non-negative $\lambda$, the norm $|f|_\lambda$ of function $f:\R \to \R$ is defined as 
$$
|f|_\lambda = \max\{\sup_{t\leq 0}e^{-\lambda t}|f(t)|,\ \sup_{t \geq 0}|f(t)|\}. 
$$ 
If we set $\eta_\lambda(t) = \min\{e^{\lambda t}, 1\}$ then clearly
$$
|f|_\lambda = \sup_{t \in \R}{|f(t)|}/{\eta_\lambda(t)}.
$$
The main result of this paper is  the following 
\begin{theorem}\label{MER}
Assume that the  initial function $w_{0}$  satisfies  the hypotheses $(IC1)$, $(IC2)$ and that, for some
$A >0$ and $c> c_*$, it holds 
$$
\lim_{x\to -\infty} w_0(s,x)
e^{-\lambda_1(c)(x+cs)} =A
$$
uniformly on $s\in[-h,0]$. If, in addition, the birth function $g$ is strictly increasing and satisfies  {\rm \bf(H)}, then 
the solution of (\ref{e1}),
(\ref{e2}) satisfies 
\begin{equation}\label{RoC}
\lim_{t\rightarrow\infty}\sup_{x \in \R}\frac{|\phi(x+ct+a)-u(t,x)|}{\eta_{\lambda_1}(x+ct)}=0,
\end{equation}
where $a=(\lambda_1(c))^{-1}\ln A$ and the front profile $\phi$ (existing in virtue of the assumption $c> c_*$) is normalised by  $\lim_{x\to-\infty}e^{-\lambda_1(c)x}\phi(x)=1$.
\end{theorem}
Theorem  \ref{MER} allows to answer the velocity selection question  for solutions with  initial data possessing  
exponential decay at $-\infty$. Indeed, suppose that, for some $\lambda >0$, it holds
\begin{equation}\label{sps}
\lim_{x\to -\infty} w_0(s,x)
e^{-\lambda x} =A(s) >0, \ \mbox{uniformly in } s \in [-h,0]. 
\end{equation}
Then define $c(\lambda)$ by the  formula $c(\lambda) = \mu/\lambda$, where $\mu$ is the unique positive root of the equation 
$$
\lambda^2 -\mu -1 +g'(0)e^{-\mu h} =0. 
$$
It is easy to see that $c(\lambda) \geq c_\#,$ where $c_\#= c_\#(g'(0),h)$ is the so-called critical speed (a uniquely determined value of $c$ for which the characteristic function $\chi_0(\lambda)$ has a double positive zero).  Set $\lambda_*:=\lambda_1(c_*)$.  We claim that 
\begin{eqnarray*} \hspace{-1cm}&&  c_\lambda:=\left\{
\begin{array}{ll} c(\lambda), & \mbox {if } \ \lambda < \lambda_*, \\
    c_*, &  \mbox {if } \ \lambda \geq \lambda_*,
\end{array}%
\right.
\end{eqnarray*}
is  the speed of propagation selected  by  solutions with  initial data satisfying (\ref{sps}). More precisely, the following assertion holds.
\begin{corollary} \label{Cor1} Assume that the initial function $w_{0}$  satisfies  the hypotheses $(IC1)$, $(IC2)$ and $(\ref{sps})$.  Suppose first that $\lambda  > \lambda_*$ and $c_* >c_\#$, then the solution of (\ref{e1}),
(\ref{e2}) satisfies 
$$
\lim_{t\to+\infty}\sup_{x \in \R}\frac{|\phi_*(x+c_*t)-u(t,x)|}{\eta_{\nu}(x+c_*t)} =0
$$
for each fixed $\nu \in (\lambda_*, \lambda)$. 
Here $\phi_*$ denotes the profile of appropriately shifted unique minimal (pushed) front to equation (\ref{e1}). 

Next, let $\lambda < \lambda_*$  (so that $c(\lambda) =c_\lambda$) and $c_* \geq c_\#$.  Set 
$$
a_-:= \frac{1}{\lambda}\ln\left[\min
_{s\in [-h,0]}A(s)e^{-\mu s}\right] \leq a_+:= \frac{1}{\lambda}\ln \left[\max
_{s\in [-h,0]}A(s)e^{-\mu s}\right]. $$
Then for every $\epsilon >0$ there exists $T_{1}({\epsilon}) >0$ such that 
$$
(1-\epsilon)\phi_\lambda(x+c_\lambda t+a_-) \leq u(t,x) \leq (1+\epsilon)\phi_\lambda(x+c_\lambda t+a_+), \quad t \geq T_{1}{(\epsilon)}, \ x \in \R. 
$$
Here $\phi_\lambda$ denotes the profile of the unique  wavefront to equation (\ref{e1}) propagating with the velocity $c(\lambda)$ and satisfying $
\lim_{x\to-\infty}e^{-\lambda x}\phi_\lambda(x)=1$.

Now, if $\lambda = \lambda_*$ and $c_* >c_\#$, then  there exists  $a' \in \R$ such that
 for every $\epsilon >0$ and positive $\nu <\lambda_*<M< \lambda_2(c_*)$ it holds 
 \begin{equation}\label{lc}
\phi_*(x+c_*t+ a') - \epsilon \eta_{M}(x+c_*t) \leq u(t,x) \leq (1+\epsilon)\phi_\nu(x+c_\nu t), \ t \geq T_{2}, \ x \in \R,
\end{equation}
for an appropriate $T_2=T_{2}(\epsilon) >0$.  
Furthermore, in such a case, $u(t,x)$ can not converge, uniformly on $\R$, to a wavefront solution of equation   (\ref{e1}). 

Finally, if $L_g= g'(0)$ (so that $c_*=c_\#$) and  $\lambda \geq \lambda_*$,  then there exists  $b' \in \R$ such that for every $\epsilon >0$ 
\begin{equation}\label{lc2}
0 \leq u(t,x) \leq (1+\epsilon)\phi_*(x+c_* t+b'), \quad t \geq T_{3}, \ x \in \R, 
\end{equation}
whenever $T_3=T_{3}(\epsilon) >0$ is sufficiently large. 
\end{corollary}
\vspace{1mm}
It is worth to note that there is an important difference between the speed selection results obtained in the non-delayed and delayed cases. Indeed, if $h=0$ and $\lambda < \lambda_*$ then $a_-=a_+$ and therefore $u(t,x)$ converges to a single 
wavefront $\phi_\lambda(x+c_\lambda t+a_\pm)$ propagating with the velocity $c_\lambda = \lambda + (g'(0)-1)/\lambda$.  In the delayed case, however, we only can say that $u(t,x)$ evolves between two shifted traveling fronts, both of them moving with the  same velocity $c_\lambda$.  Observe also that, since $\mu=\mu(h)$ is a decreasing function of $h$, the  inclusion of delay in problems modeled by (\ref{e1}) slows down the propagation of `concentrations' having the same initial distribution which satisfies  (\ref{sps}). 
\begin{remark}  Consider again the final statement of Corollary \ref{Cor1}. Under conditions assumed in it (at least when additionally $\lambda > \lambda_*$), it is natural to expect \cite{ES} the so-called {\it convergence in form} of $u(t,x)$ to the minimal 
wavefront: that is 
$$
\sup_{x\in \R}|u(t,x)-\phi_*(x+c(t))| \to 0, \ \mbox{as} \ t \to +\infty,
$$
for an appropriate function $c(t)$. Then (\ref{lc2}) implies that the function $c(t) - c_*t$ is bounded from above: in other words, in such a case,  the concentration $u(t,x)$ should  propagate behind the minimal front. A more detailed analysis  of this phenomenon for some delayed reaction-difusion models will be given in the forthcoming work by the authors. 
\end{remark}

Another immediate consequence of Theorem  \ref{MER} is  the following assertion concerning the global asymptotic stability (without asymptotic phase) of wavefronts: 
\begin{corollary}\label{cor2} Let $g$ and $w_0$ satisfy the assumptions {\rm \bf(H)} and $(IC1)$, $(IC2)$. If $g$ is strictly increasing  and 
\begin{equation}\label{sneg}
\sup_{s \in [-h,0]}|\phi(\cdot+cs)-w_0(s,\cdot)|_{\mu} < \infty
\end{equation}
for some $c> c_*$ and $\mu >\lambda_1(c)$, then the solution of (\ref{e1}),
(\ref{e2}) satisfies 
\begin{equation*}
\lim_{t\rightarrow\infty}\sup_{x \in \R}\frac{|\phi(x+ct)-u(t,x)|}{\eta_{\lambda_1}(x+ct)}=0.
\end{equation*}
\end{corollary}
Clearly, the statement of  Theorem \ref{MER} (or Corollary \ref{cor2}) implies the uniqueness (up to a translation) of non-critical traveling  fronts propagating with the same velocity $c$ and having the same order of exponential decay at $-\infty$, cf. e.g. \cite[Theorem 1.1]{MaZou}, \cite[Corollary 4.9]{WLR}.  In any event, the uniqueness of each front (including critical one) to the monotone model  (\ref{e1}) was established in \cite[Theorem 1.2]{TPT} by means of the Berestycki-Nirenberg method of the sliding solutions. In the case when $g$ is non-monotone, the 
wave uniqueness  was investigated in \cite{AGT}, by applying a suitable $L^2$-variant of the bootstrap argument suggested by  Mallet-Paret in \cite{FA}. We recall here  that, in the case of a unimodal birth-function $g$, equation (\ref{e1}) can possess non-monotone wavefronts (either slowly oscillating or eventually monotone). This fact was deduced in  \cite{IGT,FGT,TT} from the seminal results \cite{morse,FA,mps,mps2} by Mallet-Paret and Sell.  

It is instructive to compare Theorem  \ref{MER} and Corollary \ref{cor2} with the  corresponding results from the above mentioned works    \cite{MLLS,MLLS2,MeiI,MeiF,WLR} (restricting them to the particular family of the Mackey-Glass type diffusive equations (\ref{e1})).   It is easy to check that Theorem  \ref{MER} amplifies  
 Theorem 4.1 from \cite{WLR} which was proved under more restrictive smoothness and geometric conditions on $g$ and $w_0$. (Theorem \ref{Thm2}A below also extends the mentioned result by Wang {\it et al.} for the critical case $c=c_\#$). In particular, the assumptions of \cite{WLR} contain the  inequality $g'(x) \leq g'(0),$ $x\geq 0$, which excludes  from consideration the pushed waves, see \cite[Subsection 1.2]{TPT} for more detail. The approach 
 of \cite{WLR} is  a version of the super- and sub-solutions method proposed in \cite{CG1} and then further developed in \cite{MaZou}.  The proofs given in the present paper are also based on the  squeezing technique and the Phragm\`en-Lindel\"of principle  for  reaction-diffusion equations. Hence, we are also using adequate super- and sub-solutions (which generally are not $C^1$-smooth and are simpler than those considered in \cite{CG1,MaZou,WLR}. In particular, the latter fact allows  to shorten  the proofs).  
 
 Another important approach to the wave stability problem in (\ref{e1}) is a weighted energy  method  developed  by Mei {\it et al.} \cite{MLLS,MLLS2,MeiI,MeiF}. See also Kyrychko {\it et al.}  \cite{KGB}, Lv and Wang \cite{LvW}, Wu {\it et al.} \cite{WZL}.  This method is based on  rather technical weighted energy estimations and generally  requires better properties from $g$ and $w_0$. For instance, it was assumed in \cite{LvW,MeiI} that $g''(x) \leq 0,$ $x \geq 0$, and that the weighted  initial perturbation $\delta(s,x) = (\phi(x+cs)-w_0(s,x))/\eta_\mu(x)$ belongs to the Sobolev space $H^1(\R)$ for some $\mu > \lambda_1$ and for each fixed $s \in [-h,0]$. It was also assumed  in \cite{LvW,MeiI}  that $\delta: [-h,0]\to H^1(\R)$ is a continuous function that implies immediately the fulfilment of (\ref{sneg}), in virtue of the corresponding embedding theorem. Therefore  Corollary \ref{cor2} can be also used in such a situation. However, in difference with Corollary \ref{cor2}, the  weighted energy  method allows to prove the {\it exponential} stability of non-critical traveling fronts. Consequently, it gives the same convergence rates  as the Sattinger  functional analytical approach \cite{STG} gives in the case of non-delayed version of  
(\ref{e1}). We recall that the latter approach is  based on the spectral analysis of  equation (\ref{e1}) linearised  along a wavefront.  Thus a certain disadvantage of  Theorem \ref{MER} as well as  \cite[Theorem 2]{CG1}, \cite[Theorem 5.1]{MaZou}, \cite[Theorem 4.1]{WLR}  is that they do not give any estimation of the rate of convergence in (\ref{RoC}).  
In this regard, it is a remarkable fact that  super-  and sub-solutions used in this work  are also suitable to provide rather short proofs  of the exponential  stability [asymptotical stability] of non-critical [respectively, critical] wavefronts in equation (\ref{e1}) considered with the monotone birth function $g$ satisfying relatively weak restrictions  (\textbf{H}) and 
$L_g= g'(0)$. For example, $L_g= g'(0)$ if
 $g$ is differentiable on $\R_+$ where 
$g'(x)\leq g'(0)$.
\begin{theorem}\label{Thm2}
In addition to (\textbf{H}), suppose that  $g$ is strictly increasing  and
$L_g= g'(0)$. If the initial function $w_0$ satisfies the assumptions (IC1), 
(IC2),  then the solution $u(t,x)$ of (\ref{e1}),
(\ref{e2}) satisfies the following.
\begin{itemize}
\item[A.]  
If $c\geq  c_\#$ and 
\begin{equation*}\label{AS}
\lim_{z\to -\infty}w_0(s,x)/\phi(x)= 1,
\end{equation*}
uniformly on $s \in [-h,0]$, then  
\begin{equation}\label{che}
|u(t,\cdot)/\phi(\cdot+ct)- 1|_0 =o(1), \ t \to +\infty.
\end{equation}
\item[B.]   If  $c>c_\#$ and $\lambda \in (\lambda_1(c), \lambda_2(c))$ then 
\begin{equation}\label{nevad}
\sup_{s \in [-h,0]}|\phi(\cdot+cs) -w_0(s,\cdot)|_\lambda < \infty
\end{equation}
implies that 
$$\sup_{x \in \R}\frac{|u(t,x)- \phi(x+ct)|}{\eta_{\lambda}(x+ct)}  \leq Ce^{-\gamma t }, \ t \geq 0,$$
for some
$C  >0$ and $\gamma >0$.
\end{itemize}
\end{theorem}
To the best of our knowledge, the description of front convergence in the form (\ref{che}) was proposed by Chen and Guo \cite{CG1}. Clearly, this kind of convergence  is equivalent to the weighted convergence expressed by (\ref{RoC}) (if $c > c_\#$) and it is stronger than the uniform convergence 
$$
\sup_{x \in \R}|u(t,x) - \phi(x+ct)| \to 0, \ t \to +\infty. 
$$
The stability results stated in Theorem \ref{Thm2} have the global character in the sense that none smallness restriction is imposed on the norm (\ref{nevad}) of  perturbation $\phi(x+cs) -w_0(s,x)$.  Remarkably, in the case where we do not assume anymore that  $g$ is monotone,  our approach still  allows us  to prove the local stability of fronts.   Even more, we are also able to present some global stability results. In this way, our next main theorem and its corollary can be regarded as a further development of  \cite[Theorem 2.1]{LLLM} and \cite[Theorems 2.4 and 2.6]{WZL}. Before formulating the corresponding assertions, let us recall that the hypothesis
\vspace{1mm}

\noindent  {\rm \bf(UM)}  Let (\textbf{H}) be satisfied  and suppose that  
$L_g= g'(0)$ and $g$ is bounded

\vspace{1mm}
\noindent  implies the existence of a unique normalised (at $-\infty$) positive semi-wavewfront $u(t,x)=\phi_c(x+ct)$ to equation (\ref{e1}) for each $c \geq c_\#$, see e.g. \cite{AGT,SEDY}. We recall here that  the definition of a semi-wavewfront is similar to the definition of a wavefront: the only part that is changing is the boundary condition $\phi_c(+\infty) =\kappa$ which should be replaced with $\liminf_{x\to +\infty}\phi_c(x) >0$. 

\begin{theorem}\label{Thm3}
Assume (\textbf{UM}) and let the initial function $w_0$ satisfy  (IC1).  Consider  $c > c_\#,$ $\lambda \in (\lambda_1(c), \lambda_2(c))$ and set $\xi(x,\lambda)= e^{\lambda x}$.  Then the following holds. 
\begin{itemize}
\item[A.] The inequality (\ref{nevad}) implies that the solution $u(t,x)$ of (\ref{e1}),  (\ref{e2}) converges to the semi-wavefront $\phi_c(x+ct)$: more precisely, there are positive $C, \gamma$ such that 
$$\sup_{x \in \R}\frac{|u(t,x)- \phi(x+ct)|}{\xi(x+ct,\lambda)}  \leq Ce^{-\gamma t }, \ t \geq 0.$$
\item[B.]  Let, in addition,  $|g'(u)| <1$ on some interval
$[\kappa-\rho, \kappa + \rho]$, $\rho>0$.  If, for some $b\geq 0$, 
the initial function $w_0$ and the semi-wavefront profile $\phi_c$ satisfy$$
\kappa-\rho/4 \leq w_0(s,x),\phi(x+cs) < \kappa + \rho/4 \ \mbox{for all}\ x\geq b-ch, \ s \in [-h,0],
$$
$$
|w_0(s,x)-\phi(x+cs)| \leq 0.5\rho e^{\lambda(x+cs-b)},\  x\leq b, \ s \in [-h,0], 
$$
then $\phi$ is actually a wavefront (i.e. $\phi(+\infty)=\kappa$) and the solution $u(t,x)$ of (\ref{e1}),  (\ref{e2}) satisfies 
\begin{equation}
\label{gc}
\sup_{x \in \R}\frac{|u(t,x)- \phi(x+ct)|}{\eta_{\lambda}(x+ct)}  \leq 0.5\rho e^{-\gamma t }, \ t \geq 0,
\end{equation}
for some $\gamma >0$. 
\end{itemize}
\end{theorem}
\begin{corollary}\label{co3}
Let $g$ satisfy (\textbf{UM})  and let $g$ be a unimodal function, with a unique point $x_m\in (0,\kappa)$ of local extremum (maximum).  Suppose further that $|g'(x)|<1$ for all $x \in [g(g(x_m)), g(x_m)]$. Additionally, assume  that 
the initial function $w_0$ satisfy  (IC1) and (IC2) and consider  $c > c_\#,$ $\lambda \in (\lambda_1(c), \lambda_2(c))$.  Then inequality (\ref{nevad}) implies that the solution $u(t,x)$ of (\ref{e1}),  (\ref{e2}) uniformly converges to the wavefront $\phi(x+ct)$. More precisely, there are positive $\rho, \gamma$ such that (\ref{gc}) holds. 
\end{corollary}
Let us illustrate Corollary \ref{co3} by considering the well-known diffusive version of the Nicholson's blowflies equation 
\begin{eqnarray} \label{nb}
u_{t}(t,x) &=& u_{xx}(t,x) -\delta u(t,x) + pu(t-\tau,x)e^{-u(t-\tau,x)}.  
\end{eqnarray}
By rescaling space-time coordinates, we transform this equation into the  form (\ref{e1}) with $g(x) = (p/\delta)xe^{-x}$ and  $h = \tau \delta$. In the last decade, 
the wavefront solutions of  equation (\ref{nb}) have been investigated by many authors, e.g. see \cite{CMYZ,IGT,FGT,LLLM,LvW,ma,MLLS,MeiI,MeiF,WLR,WZL}.  If the positive parameters $p,\delta$ are such that 
$1< p/\delta \leq e$, then $g$ is monotone and satisfies the hypothesis  {\rm \bf(H)}  with $L_g = g'(0)$ and 
$\kappa = \ln (p/\delta)$. In such a case, Theorem \ref{Thm2} guarantees the global stability of all wavefronts, including the minimal one (these wavefronts are necessarily monotone). For the first time, such a  global stability result  was established by Mei {et al.} in \cite{MeiI}.  Now, if $e< p/\delta < e^2$, the restriction of $g$ on $[0,\kappa]$ is not monotone anymore. Nevertheless, we still have that $L_g=g'(0)$ while the inequality $|g'(x)|<1$ holds for all $x \in [g(g(x_m)), g(x_m)]$, with $x_m=1$. Therefore, for each $p/\delta \in (e,e^2)$,  Corollary \ref{co3} assures the global exponential stability of all non-critical wavefronts to equation (\ref{nb}). Note that profiles of these wavefronts are not necessarily monotone and they can either slowly oscillate   around $\kappa$ or be  non-monotone but eventually monotone at $+\infty$, cf. \cite{IGT,FGT,LLLM}.  Observe also that the upper estimation $e^2$ for $p/\delta$ is optimal \cite{FGT,LLLM}. Under the same restriction $p/\delta \in (e,e^2)$,  the local stability of wavefronts to (\ref{nb}) was investigated in \cite{LLLM,WZL}.

To sum up: the main aim of the present work is to establish the stability properties of monostable wavefronts  
to the time-delayed reaction-diffusion model (\ref{e1}) with generally non-convex and non-smooth birth function $g$.  We are going to achieve this goal  by developing  suitable  ideas and methods from \cite{AW,FML,sch,STR,U}.

Finally, let us say a few words about the organization of the paper.  In Sections  \ref{sub2} and \ref{Sec4} we prove several auxiliary  comparison and stability results. Then
Theorem \ref{MER} (with Corollary \ref{Cor1}), Theorem  \ref{Thm2}  and Theorem \ref{Thm3} (with Corollary \ref{co3}) are proved in Sections \ref{sub4}, \ref{sub3}  and \ref{sub6}, respectively.

\section{Super- and sub-solutions: definition and properties} \label{sub2} 
The stability analysis of a wavefront $u=\phi(x+ct)$ is usually realised  in the co-moving coordinate frame  $z=x+ct$ so that 
$w(t,z):=u(t,z-ct)=u(t,x)$. Clearly,  $w$ satisfies the equation
\begin{eqnarray}\label{E1}
w_{t}(t,z)=w_{zz}(t,z)-cw_{z}(t,z)-w(t,z)+g(w(t-h,z-ch)), \end{eqnarray}
while the front profile $\phi(z)$ is a solution of the stationary equation 
\begin{eqnarray}\label{EP}
0=\phi''(z)-c\phi'{z}-\phi(z)+g(\phi(z-ch)). \end{eqnarray}
In order to study the front solutions of (\ref{E1}), (\ref{EP}),
different versions of the method of super- and sub- solutions were successfully applied 
in \cite{ma,sch,TPT,wz} (in the case of stationary equations similar to (\ref{EP})) and in 
\cite{CG1,MaZou,sch,STR,WLR} (in the case of non-stationary equations similar to (\ref{E1})).  
An efficacious construction of these solutions  is the key to the success of this approach. 
In particular,  the studies of  front's stability in \cite{MaZou,WLR} had used $C^3$-smooth super- and sub-solutions  previously introduced  by Chen and Guo in \cite[ Lemma 3.7]{CG1}. It is well known that, by cautiously  weakening smoothness restrictions, we can improve the overall quality of super- and sub- solutions, cf. \cite{FML,ma,roth,sch,TPT,WLR,wz}. In this paper, inspired by the latter references, we propose to work with somewhat more handy $C^1$-smooth super- and sub-solutions:  
\begin{definition}\label{def1}
Continuous function $w_+:  [-h, +\infty) \times \R \to \R$
is called a super-solution for (\ref{E1}), if, for some $z_* \in
\R$, this function  is $C^{1,2}$-smooth in the domains $[-h, +\infty) \times
(-\infty, z_*]$ and $[-h, +\infty) \times [z_*, +\infty)$ and, for every $t >0$,
\begin{equation}  \label{sso} \hspace{-0.5cm} {\mathcal N}w_+(t,z) \geq 0,  \ z \not= z_*,  \ \mbox{while} \
(w_+)_z(t,z_*-)> (w_+)_z(t,z_*+), 
\end{equation}
where the nonlinear operator ${\mathcal N}$ is defined by
\begin{equation} \hspace{-.0cm}
{\mathcal N}w(t,z):= w_{t}(t,z)
-w_{zz}(t,z)+cw_{z}(t,z)+w(t,z)-g(w(t-h, z-ch)). \nonumber
\end{equation}
The definition of  a  sub-solution $w_-$ is similar, with the  inequalities reversed in 
(\ref{sso}). 
\end{definition}
The following comparison result is a rather standard one. However, since  sub- and super-solutions considered in this paper have discontinuous spatial derivates and, in addition, equation (\ref{E1}) contains shifted arguments, we give its proof for the completeness of our exposition. See also \cite{FML,roth,sch,WLR}. 
\begin{lemma} \label{cl} Assume {\rm \bf(H)} and the monotonicity of $g$. Let $w_+, w_-$ be a pair of super- and sub-solutions for equation (\ref{E1}) such that  $|w_\pm(t,z)| \leq Ce^{D|z|}$, $t \geq -h, \ z \in \R$, for some $C, D >0$ as well as
$$
w_-(s,z) \leq w_0(s,z) \leq w_+(s,z), \quad \mbox{for all} \ s \in [-h,0], \ z \in \R.  
$$
Then the solution $w(s,z)$  of  equation (\ref{E1}) with the initial datum $w_0$  satisfies 
$$
w_-(t,z) \leq w(t,z) \leq w_+(t,z) \quad \mbox{for all} \ t \geq -h, \ z \in \R.  
$$
\end{lemma}
\begin{proof} In view of the assumed conditions, we have that
$$
\pm(g(w_\pm(t-h,z-ch))- g(w(t-h,z-ch))) \geq 0, \quad t \in [0,h], \ z
\in \R.
$$
Therefore, for all  $t \in (0,h]$, the function $\delta(t,z):= \pm(w(t,z)- w_\pm(t,z))$
satisfies the inequalities
\begin{eqnarray}
& &  \hspace{-0cm} \delta(0,z) \leq 0,\ |\delta(t,z)| \leq 2Ce^{D|z|}, \quad
\delta_{zz}(t,z)- \delta_{t}(t,z)-c\delta_{z}(t,z)-\delta (t,z) = \nonumber \\
& &  \hspace{-0cm} \pm ({\mathcal N}w_\pm(t,z) - {\mathcal N}w(t,z) +
g(w_\pm(t-h, z-ch)) -
 g(w(t-h, z-ch))) =\nonumber  \\  & &  \hspace{-0cm} \pm{\mathcal N}w_\pm(t,z)\pm (g(w_\pm(t-h, z-ch)) -
 g(w(t-h, z-ch)))
   \geq 0,  \  z \in \R\setminus\{z_*\}; \nonumber 
 \end{eqnarray}   
 \begin{equation}
\frac{\partial \delta(t,z_*+)}{\partial z}- \frac{\partial \delta(t,z_*-)}{\partial z}=  \pm\left(\frac{\partial w_\pm(t,z_*-)}{\partial z} -\frac{\partial w_\pm(t,z_*+)}{\partial z}\right) > 0. \label{di}
\end{equation}
We claim that $\delta(t,z) \leq 0$ for all $t \in [0,h], \ z \in
\R$. Indeed, otherwise there exists $r_0> 0$ such that
$\delta(t,z)$ restricted to any rectangle $\Pi_r= [-r,r]\times
[0,h]$ with $r>r_0$,   reaches its maximal positive value $M_r >0$
at  at some point $(t',z') \in \Pi_r$.

We claim  that $(t',z')$ belongs to the parabolic boundary
$\partial \Pi_r$ of $\Pi_r$. Indeed, suppose on the contrary, that
$\delta(t,z)$ reaches its maximal positive value at some point
$(t',z')$ of $\Pi_r\setminus \partial \Pi_r$. Then clearly $z'
\not=z_*$ because of (\ref{di}). Suppose, for instance that $z' >
z_*$. Then $\delta(t,z)$ considered on the subrectangle $\Pi=
[z_*,r]\times [0,h]$ reaches its maximal positive value $M_r$ at the
point $(t',z') \in \Pi \setminus \partial \Pi$.  Then the
classical results \cite[Chapter 3, Theorems 5,7]{PW} show that
$\delta(t,z) \equiv M_r >0$ in $\Pi$, a contradiction.

Hence, the usual maximum principle holds for each $\Pi_r, \ r \geq
r_0,$ so that we can appeal to the proof of the
Phragm\`en-Lindel\"of principle from \cite{PW} (see Theorem 10 in
Chapter 3 of this book), in order to conclude that  $\delta(t,z)
\leq 0$ for all  $t \in [0,h], \ z \in \R$.

But then we can again repeat the above argument on the intervals $[h,2h],$ $[2h, 3h], \dots$ establishing that the inequality $w_-(t,z) \leq w(t,z)\leq w_+(t,z),$  $z\in \R,$ holds for all $t \geq -h$.  
\qed
\end{proof}
To the best of our
knowledge, the following important  property of super- (sub-) solutions was first used by  Aronson and  Weinberger in \cite{AW}. See also \cite[Proposition 2.9]{sch}.
\begin{corollary} \label{cod1} Assume {\rm \bf(H)} and the monotonicity of $g$. Let $w_+(z)$ be an exponentially bounded super-solution for equation (\ref{E1}) and consider the solution $w^+(t,z), t \geq 0,$ of the initial value problem  $w^+(s,z)= w_+(z)$ for (\ref{E1}). 
Then $w^+(t_1,z) \geq w^+(t_2,z)$  for each $t_1 \leq t_2,$ $z \in \R$. A similar result is valid in the case of  exponentially bounded sub-solutions  $w_-(z)$ which do not depend on  $t$: if $w^-(t,z)$ solves the initial value problem  $w^-(s,z)= w_-(z)$ for (\ref{E1}), then 
 $w^-(t_1,z) \leq w^-(t_2,z)$  for each $t_1 \leq t_2,$ $z \in \R$.
\end{corollary}
\begin{proof} We prove only the first statement of the corollary (for super-solution  $w_+$), the case of sub-solution $w_-(z)$ being completely analogous.  

By Lemma \ref{cl}, $w^+(t,z) \leq w_+(z)$ for each $t\geq 0$. Hence, fixing some positive $l$ and considering the initial value problems 
$
u(s,z) = w^+(s+l,z),  \ v(s,z) = w_+(z),$ $s \in [-h,0], \ z \in \R, 
$
for equation (\ref{E1}), 
we find that $u(t,z)= w^+(t+l,z) \leq v(t,z) = w^+(t,z)$, $t > 0, \ z \in \R$.   \qed
\end{proof}
\section{Proof of Theorem \ref{Thm2} and Corollary \ref{Cor1}}\label{sub3}

In this section, we take some $c\geq c_\#$ and assume the conditions of Theorem \ref{Thm2}. This result will follow from Theorem \ref{Te3} proved below. Everywhere in the section we denote by $w(t,z)$ solution  of equation (\ref{E1}) satisfying the initial value condition $w(s,z) = w_0(s,z),$ $(s,z) \in \Pi_0$. 

It is easy to see that, 
given $q^* >0,\ q_* \in (0,\kappa)$, there are $ \delta^* < \delta_0$, $\gamma^* >0$  such that
\begin{eqnarray} \label{gg1}
\begin{array}{ll}
g(u)- g(u- qe^{\gamma h}) \leq q(1-2\gamma), \\
(u,q,\gamma)\in \Pi_-= [\kappa-\delta^*,\kappa]\times[0,q_*] \times
 [0,\gamma^*];
\end{array}
\end{eqnarray} 
\begin{eqnarray} \label{gg}
\begin{array}{ll}
g(u)- g(u+ qe^{\gamma h}) \geq - q(1-2\gamma), \\
(u,q,\gamma)\in \Pi_+= [\kappa-\delta^*,\kappa]\times[0,q^*] \times
 [0,\gamma^*].
\end{array}
\end{eqnarray} 
Indeed, it suffices to note that
the continuous functions
\begin{eqnarray*} \hspace{-0.1cm}&&  G_-(u,q, \gamma):=\left\{
\begin{array}{ll} 1+(
g(u- e^{\gamma h}q)-& g(u))/q, \  (u,q,\gamma)\in \Pi_-; \\
    1- e^{\gamma h }g'(u), & {u\in[\kappa-\delta^*,\kappa], \ q =0},\ \gamma \in [0,\gamma^*],
\end{array}%
\right.
\end{eqnarray*}
\begin{eqnarray*} \hspace{-0.1cm}&&  G_+(u,q, \gamma):=\left\{
\begin{array}{ll} 1-(
g(u+e^{\gamma h}q)-& g(u))/q, \  (u,q,\gamma)\in \Pi_+; \\
    1- e^{\gamma h }g'(u), & {u\in[\kappa-\delta^*,\kappa], \ q =0},\ \gamma \in [0,\gamma^*],
\end{array}%
\right.
\end{eqnarray*}
are positive on $\Pi_\pm$ provided that $\gamma^*, \delta^*$   are
sufficiently small. 

From now on, we fix  $\gamma \in [0, \gamma_*), $ $\delta \in (0,\delta^*)$ such that 
(\ref{gg1}) and (\ref{gg}) hold and $$-\gamma+c\lambda-\lambda^{2}+1-g'(0)e^{\gamma
h}e^{-\lambda ch}\geq 0.$$ It is easy to see that  $\gamma =0$ for
$\lambda = \lambda_1(c)$ while $\gamma$ can be chosen positive if 
$\lambda \in (\lambda_1(c), \lambda_2(c))$. Consider  $b$ determined by the equation 
$\phi(b-ch)=\kappa-\delta^*/2$. Without loss of generality we can assume that $b>0$. 
\begin{lemma}\label{Sttg}
Suppose that  $L_g=g'(0)$ in $(\textbf{H})$. Let $\gamma \geq 0$ be as defined above. If either $c>c_\#$ with $\lambda \in (\lambda_1(c), \lambda_2(c))$ or $c\geq c_\#$ with $\lambda = \lambda_1(c)$, then 
\begin{equation*}\label{mlem}w_0(s,z)\leq \phi(z)+q \eta_\lambda(z-b), \quad z\in\mathbb{R},\quad s\in [-h,0],
\end{equation*}
with $q \in (0,q^*]$ implies
 \begin{equation*}\label{mlemm}w(t,z)\leq \phi(z)+qe^{-\gamma t}
\eta_\lambda(z-b), \quad z\in\mathbb{R},\quad t\geq -h.
\end{equation*}
Similarly, the inequality
\begin{equation*}\label{mlemN}\phi(z)-q \eta_\lambda(z-b) \leq w_0(s,z),  \quad z\in\mathbb{R},\quad s\in [-h,0],
\end{equation*}
with some $0< q \leq q_*$ implies
 \begin{equation*}\label{mlemmN}
 \phi(z)-qe^{-\gamma t} \eta_\lambda(z-b) \leq w(t,z),
\quad z\in\mathbb{R},\quad t\geq -h.
\end{equation*}
Each conclusion of the lemma  holds  without any upper restriction on the size of $q$ if we replace $\eta_\lambda(z-b)$ with $\xi(z, \lambda) = \exp{(\lambda z)}$. 
\end{lemma}
\begin{proof} 
Set $w_{\pm}(t,z)=\phi(z)\pm qe^{-\gamma
t}\eta_\lambda(z-b)$. Then, for $t>0$ and $z\in\R\setminus\{b\}$, after a direct calculation we find that 
$$
\mathcal{N}w_{\pm}(t,z)= \pm qe^{-\gamma
t}[-\gamma\eta_\lambda(z-b)+c\eta'_\lambda(z-b)-\eta''_\lambda(z-b)+\eta_\lambda(z-b)]+$$
$$g(\phi(z-ch))-g(w_{\pm}(t-h,z-ch)). 
$$
It is clear that for $z< b$ (if we are considering  $\eta_\lambda(z-b)$) as well as for all $z\in \R$ (if  we are using $\xi(z,\lambda)$ instead of $\eta_\lambda(z-b)$),  it holds that
$$
\pm\mathcal{N}w_{\pm}(t,z)\geq  qe^{-\gamma
t}e^{\lambda(z-b)}[-\gamma+c\lambda-\lambda^{2}+1-g'(0)e^{\gamma
h}e^{-\lambda ch}]\geq 0.
$$
If $z> b$ and $q \in (0, q^*]$, then (\ref{gg})  implies 
\begin{eqnarray*}
\mathcal{N}w_{+}(t,z)&\geq & qe^{-\gamma
t}[-\gamma+1- (1-2\gamma)] = \gamma q e^{-\gamma
t} >0.  
\end{eqnarray*}
Similarly, 
if $z> b$ and $q \in (0, q_*]$, we obtain from (\ref{gg1}) that
\begin{eqnarray*}
-\mathcal{N}w_{-}(t,z)&\geq & qe^{-\gamma
t}[-\gamma+1- (1-2\gamma)] = \gamma q e^{-\gamma
t} >0.  
\end{eqnarray*}

Next, since $$\pm\left(\frac{\partial w_\pm(t,b+)}{\partial z}- \frac{\partial
w_\pm(t,b-)}{\partial z}\right)= -  q\lambda e^{-\gamma t} <0,$$ 
we conclude that $w_{\pm}(t,z)$ is a pair of super- and sub-solutions for equation (\ref{E1}). 
Finally, an application of Lemma \ref{cl} completes the proof.   \hfill $\square$
\end{proof}
Lemma \ref{Sttg} implies that front solutions of equation (\ref{e1}) are locally stable:
\begin{corollary} \label{simp} Let the triple $(c,\lambda, \gamma)\in [c_\#, +\infty)\times  [\lambda_1(c),  \lambda_2(c))\times \R_+$ be as in Lemma \ref{Sttg} and suppose that 
$$\sup_{s \in [-h,0]}|\phi(\cdot) -w_0(s,\cdot)|_{\lambda} < \rho e^{-\lambda b}$$ 
for some  $\rho< \kappa$. Then  
$$|\phi(\cdot) -w(t,\cdot)|_{\lambda} < \rho e^{-\gamma t}, \ t \geq 0.$$
\end{corollary}
\begin{proof} The  statement of the corollary is an immediate consequence of  Lemma \ref{Sttg},  since, due to our assumptions,  for all $ z \in \R,\  s \in [-h,0]$, 
$$\phi(z) - \rho \eta_{\lambda}(z-b) \leq \phi(z) - \rho e^{-\lambda b}\eta_{\lambda}(z)   \leq w_0(s,z) \leq$$
$$ \phi(z) + \rho e^{-\lambda b}\eta_{\lambda}(z) \leq \phi(z) + \rho \eta_{\lambda}(z-b).\qquad \qquad \qquad {\qed}
$$
\end{proof}

We note that assumption $(IC1)$ allows consideration of initial functions $w_0$ which can be equal to $0$ on compact subsets 
of $\Pi_0$. This fact complicates the construction of adequate sub-solutions.  In the next assertion we show that, without restricting generality, the positivity of $w_0$ can assumed in  our proofs. 
\begin{corollary} \label{zd} Suppose that  $L_g=g'(0)$ in $(\textbf{H})$ and that $w_0(s,z),$ $(s,z) \in \Pi_0,$ satisfies the assumptions  (IC1), (IC2). Then the following holds.
\begin{itemize}
\item[A.]   If $c\geq  c_\#$ and 
\begin{equation}\label{AS3}
\lim_{z\to -\infty}w_0(s,z)/\phi(z)= 1,
\end{equation}
uniformly on $s \in [-h,0]$,  then $w(2h+s,z) >0, \ (s,z) \in \Pi_0$, also  satisfies the assumptions (IC1), (IC2)  and $\lim_{z\to -\infty}w(t,z)/\phi(z)= 1$
uniformly with respect to $t \in [0,+\infty)$.   
\item[B.]  Suppose that  $c> c_\#$, $\lambda \in (\lambda_1(c), \lambda_2(c))$   together with 
\begin{equation}\label{nev}
q_0:= \sup_{s \in [-h,0]}|\phi(\cdot) -w_0(s,\cdot)|_\lambda < \infty.
\end{equation}Then $w(2h+s,z) >0, \ (s,z) \in \Pi_0$, also  satisfies the assumptions (IC1), (IC2)  and, for each $t \geq 0$, 
 $$
\sup_{s \in [-h,0]}|\phi(\cdot) -w(s+t,\cdot)|_\lambda < \infty.
$$
\end{itemize}
\end{corollary} 
\begin{proof}  The positivity of $w(2h+s,z)$ for  $(s,z) \in \Pi_0,$ is obvious. Next, the fulfilment of separation condition $(IC2)$ for $w(2h+s,z)$ can be proved similarly to \cite[Proposition 1.2]{STR} (alternatively, the reader can use Duhamel's formula). Next, since $w\equiv 0$ and $w\equiv \max\{\kappa, |w_0|_\infty\}$ are, respectively,  sub- and super-solutions of equation (\ref{E1}), the condition $(IC1)$ is also fulfilled.  Finally, the proofs of the persistence  of properties (\ref{AS3}) and (\ref{nev}) are given below. 

\noindent A.    Set $\lambda_c= \lambda_1$ if $c=c_\#$ or fix some $\lambda_c \in (\lambda_1(c), \lambda_2(c))$ if $c > c_\#$. 
It follows from (\ref{AS3}) that for every $s \in \R$, it holds 
\begin{equation*}\label{AS}
\lim_{z\to -\infty}w_0(s,z)/\phi(z+s)= e^{\lambda_1s}
\end{equation*}
uniformly on $s \in [-h,0]$.  Therefore, 
 for each small $\delta >0$ there exists a  large $q= q(\delta,w_0)>0$ such that 
\begin{equation}\label{13p}
\phi(z-\delta) - q \xi(z, \lambda_c) \leq w_0(s,z) \leq   \phi(z+\delta) + q \xi(z, \lambda_c) , \ (s,z) \in \Pi_0. 
\end{equation}
Then Lemma \ref{Sttg} assures that  
$$
\phi(z-\delta) - q \xi(z, \lambda_c) \leq w(t,z) \leq   \phi(z+\delta) + q \xi(z, \lambda_c) , \ t \geq 0,  \ z \in \R,  
$$
so that, for all $t \geq 0$ and $z \in \R$, it holds 
$$
\frak{l}(z,\delta):= \frac{\phi(z-\delta)}{\phi(z)}-1 - q \frac{\xi(z, \lambda_c)}{\phi(z)} \leq \frac{w(t,z)}{\phi(z)} -1 \leq \frak{r}(z,\delta): =\frac{\phi(z+\delta)}{\phi(z)}-1 + q \frac{\xi(z, \lambda_c)}{\phi(z)}.
$$
Now, since 
$$
\lim_{z \to -\infty} \frak{l}(z,\delta)= e^{-\lambda_1\delta}-1,\quad  \lim_{z \to -\infty} \frak{r}(z,\delta)= e^{\lambda_1\delta}-1, 
$$
for each $\epsilon >0$ we can indicate $\delta = \delta(\epsilon)$ and $z_\epsilon$ such that 
$$-\epsilon  \leq \frac{w(t,z)}{\phi(z)} -1 \leq\epsilon \quad  \mbox{for all} \  t \geq 0 \ \mbox{ and} \ z \leq z_\epsilon.$$ 

\noindent B. We have that 
$$\phi(z)-q_0 \xi(z,\lambda) \leq w_0(s,z)\leq \phi(z)+q_0 \xi(z,\lambda),  \ (s,z) \in \Pi_0,
$$
so that the last conclusion of the corollary follows from Lemma \ref{Sttg}. 
\qed
\end{proof}
\begin{remark}
Corollary \ref{zd}A shows that asymptotic relation (\ref{AS3}) is a time invariant of $w(t,z)$. In the next section, Lemma \ref{Inv}  gives an amplified version of this result. 
\end{remark}
\begin{theorem}\label{Te3}
In addition to $(\textbf{H})$, suppose that  $g$ is stictly increasing and 
$L_g= g'(0)$. If the initial function $w_0$ satisfies the assumptions
(IC1), (IC2),  then   the solution $w(t,z)$ of the initial value problem 
$w(s,z) = w_0(s,z), (s,z) \in \Pi_0,$ for (\ref{E1}) 
satisfies the following conclusions.
\begin{itemize}
\item[A.]   Take  $c\geq  c_\#$ and assume (\ref{AS3}).   Then 
$$|w(t,\cdot)/\phi(\cdot)- 1|_0 =o(1), \ t \to +\infty.$$
\item[B.]  If  $c>c_\#$ and $\lambda \in (\lambda_1(c), \lambda_2(c))$ then (\ref{nev}) implies 
$$|\phi(\cdot) -w(t,\cdot)|_\lambda \leq Ce^{-\gamma t }, \ t \geq 0,$$
for some positive $C, \gamma$ (in fact, $\gamma >0$ can be chosen  as in Lemma \ref{Sttg}). 
\end{itemize} 
\end{theorem}
\begin{proof}  In virtue of Corollary \ref{zd}, without loss of generality, we can assume that $w_0(s,z)>0$ on $\Pi_0$.  

\vspace{2mm}

\noindent A.  As in the proof of Corollary \ref{zd}A,  set $\lambda_c= \lambda_1$ if $c=c_\#$ or take some $\lambda_c \in (\lambda_1(c), \lambda_2(c))$ if $c > c_\#$. 
We know from Lemma \ref{Sttg} that the functions $\phi(z)\pm q \xi(z, \lambda_c)$ constitute  a pair of super- and sub-solutions for equation (\ref{E1}) for each positive $q$. The main  drawback of these solutions is their  unboundedness.  Hence, first we show how to correct this deficiency  of   $\phi(z)\pm q \xi(z, \lambda_c)$. 

So, fix $\delta >0$ and take $q= q(\delta,w_0)>0$ large enough to meet (\ref{13p}). 
Let  $(-\infty, p)$ be  the maximal interval where the function $\phi(z-\delta) - q \xi(z, \lambda_c)$ is positive. Then,  for sufficiently small $\epsilon \in (0, \kappa)$, the equation 
$$
\phi(z-\delta) - q \xi(z, \lambda_c) = \epsilon 
$$
has exactly two solutions $z_1(\epsilon) <z_2(\epsilon)$ on $(-\infty, p)$.  It holds that 
$z_1(0+)=-\infty,$ $z_2(0+) = p$ and therefore we can find $\epsilon>0$ such that $z_2(\epsilon) -z_1(\epsilon) >ch$ and 
$$
\inf \{w_0(s,z): z \geq z_1(\epsilon), \ s \in [-h,0]\} > \epsilon. 
$$
It is easy to see that the functions
$$
w_-(z):=\left\{\begin{array}{ll}
\phi(z-\delta) - q \xi(z, \lambda_c),& z\leq z_2(\epsilon), \\ \epsilon, & z_2(\epsilon) \leq z, \end{array}\right.$$
and 
$$w_+(z):=\min\{\kappa + |w_0|_\infty,  \phi(z+\delta) + q \xi(z, \lambda_c)\}$$  
satisfy 
$$
w_-(z)  \leq w_0(s,z)  \leq w_+(z), \ (s,z) \in \Pi_0, 
$$
and that  they are, respectively,  a sub-solution and a super-solution  for equation (\ref{E1}). 
 Thus Corollary \ref{cod1} implies that 
\begin{equation}\label{vrt}
w_-(z) \leq w^-(t,z) \leq w(t,z) \leq w^+(t,z) \leq w_+(z), 
\end{equation}
where $w^\pm(t,z)$ denote the solutions of (\ref{E1}) satisfying the initial conditions $w^\pm(s,z) =w_{\pm}(z),$ $z\in \R$, $s \in [-h,0]$.  From  Corollary \ref{cod1} we also obtain that 
$w^\pm(t,z)$ converge (uniformly on compact subsets of $\R$) to some functions $\phi^\pm(z)$ such that
$$
w^-(z) \leq \phi^-(z) \leq \phi^+(z) \leq w^+(z). 
$$
It is well known (see e.g. \cite[Lemma 2.8]{STR}) that $\phi^\pm$ satisfy the profile equation (\ref{EP}). Since 
$\phi^\pm$ are positive and  bounded, $\phi(-\infty) =0$ and $\liminf_{z\to +\infty}\phi(z) >0$, we conclude from \cite[Proposition 2 and Theorem 1.2]{TPT} that 
$\phi^\pm(z)= \phi(z\pm \delta_\pm),$ $z \in \R$ for some $-\delta \leq \delta_-\leq \delta_+ \leq \delta$. 

Furthermore, we claim that 
$$
w^*:= \limsup_{t\to +\infty}|w^+(t,\cdot)|_\infty \leq \kappa, \quad w_*:= \lim_{(T,Z)\to +\infty} \ \inf_{z \geq Z, t \geq T} w^-(t,z) = \kappa. 
$$
Clearly, $w_* \leq w^*$. To prove that $w^* \leq \kappa$, it suffices to observe that the homogeneous solution $w_g(t), \ t \geq 0,$ of equation (\ref{E1})  defined as the solution of the initial value problem 
$$
w'(t) = - w(t) + g(w(t-ch)), \ w_g(s) = |w_0|_\infty + \kappa, \ s \in [-h,0], 
$$
dominates  $w^+$ (i.e. $w^+(t,z) \leq w_g(t)$ for all $z \in \R$, $t \geq -h$) in view of Lemma \ref{cl} and converges  to $\kappa$.

Next, suppose that $w_* < \kappa$ and take $Z,T$ so large and $\delta_1 >\varepsilon_1 >0$ so small that 

(i) $w^-(t,z) >  w_* - \delta_1$ for all $z \geq Z-ch,$ \ $t \geq T-h$; 

(ii)  $w^-(t,z) > \kappa-\varepsilon_1$ for all 
$t\geq T-h,$ and $z \in [Z-ch,Z]$;

(iii) homogeneous solution $w_h(t), \ t \geq 0,$ of equation (\ref{E1})  defined as the solution of the initial value problem 
$$
w'(t) = - w(t) + g(w(t-ch)), \ w_h(s) = w_* -\delta_1, \ s \in [-h,0], 
$$
satisfies the inequalities 
\begin{eqnarray}
& &  \hspace{-0cm} w_h(t) \leq  (w_*+\kappa)/2 , \ t \in [-h, a_1], \nonumber \\
& &  \hspace{-0cm} 
(w_*+\kappa)/2 \leq w_h(t \leq  w_h(a_2) =\kappa - \varepsilon_1, \ t \in [a_1, a_2], \nonumber \end{eqnarray}
for sufficiently large $a_2 > a_1+h >h$
(observe here that from \cite[Corollary 2.2, p. 82]{HS} we know that $w_h(t)$ converges monotonically  to $\kappa$).
Therefore, for each $T_1 \geq T$ and all  $t \in (T_1,T_1+h]$, $z \geq Z$, the function  $\delta(t,z) = w_h(t-T_1) - w^-(t,z) $ satisfies the inequalities
\begin{eqnarray}
& &  \hspace{-0cm} |\delta(t,z)| \leq \kappa, \quad
\delta_{zz}(t,z)- \delta_{t}(t,z)-c\delta_{z}(t,z)-\delta (t,z) = \nonumber \\
& &  \hspace{-0cm} 
g(w^-(t-h, z-ch)) -
 g(w_h(t-T_1-h)) >0.\nonumber \end{eqnarray}
 In addition, we have that 
 $$
 \delta(T_1,z) < 0, z \geq Z, \quad  \delta(t,Z) < 0, \ t \in [T_1, T_1+a_2].
 $$
In consequence, by the Phragm\`en-Lindel\"of principle, $$\delta(t,z) = w_h(t-T_1) - w^-(t,z)\leq 0, \ \mbox{for all} \ 
t \in [T_1,T_1+h], \ z \geq Z.$$  It is clear that, using step by step integration method,  we can repeat the above procedure till the maximal moment $t_*$ before which the inequality 
$g(w^-(t-h, z-ch))\geq g(w_h(t-T_1-h))$ for $z \geq Z$ is preserved. Therefore 
$$ w_h(t-T_1) \leq  w^-(t,z) \ \mbox{for all} \ 
t \in [T_1,T_1+a_2], \ z \geq Z,$$
so that  
$$(w_*+\kappa)/2 \leq  w^-(t,z) \ \mbox{for all} \ 
t \in [T_1+a_1,T_1+a_2], \ z \geq Z.$$
However, since $T_1\geq T$ is an arbitrarily chosen number, we conclude that $w_*\geq (w_*+\kappa)/2$, contradicting to our initial assumption that $w_*<\kappa$. 
Hence 
$w^\pm(t,z) \to \phi^\pm(z)$ as $t \to +\infty$ uniformly on  $\R$.   
In virtue of (\ref{vrt}), we obtain 
$$\limsup_{t \to +\infty}|w(t,\cdot)/\phi(\cdot)- 1|_0 \leq e^{\lambda_1 \delta} -1, $$
for each small $\delta$. This  completes the proof of the first part  of  Theorem \ref{Te3}.  
\vspace{2mm}

\noindent B.  We deduce  from (\ref{nev}) that 
$$
\phi(z) - q_0e^{\lambda b} \xi(z-b,\lambda)   \leq w_0(s,z) \leq \phi(z) + q_0e^{\lambda b}\xi(z-b, \lambda), \quad z \in \R,\  s \in [-h,0]. 
$$
 As a consequence, Lemma \ref{Sttg} guarantees  that, for some positive $\gamma$ and all $z \in \R,$ $ t \geq  -h,$
$$\phi(z) - q_0 e^{\lambda b} e^{-\gamma t}\xi(z-b,\lambda)   \leq w(t,z) \leq \phi(z) + q_0e^{\lambda b} e^{-\gamma t} \xi(z-b, \lambda). $$
From the part A of this theorem, we also know that $\lim_{t\to +\infty} w(t,z) = \phi(z)$ uniformly on $\R$. 
Therefore there exist a large  $T_1>0$ and positive $q_2 < \min\{q^*,q_*\}$ such that, for all $z \in \R,$ $ t \geq  T_1-h,$
$$\phi(z) - q_2 \eta_\lambda(z-b)   \leq w(t,z) \leq \phi(z) + q_2 \eta_\lambda (z-b).
$$
Again applying Lemma \ref{Sttg}, we obtain that  
$$
\phi(z)-q_2e^{-\gamma (t-T_1)}\eta_\lambda(z-b)\leq
w(t,z)\leq\phi(z)+q_2e^{-\gamma (t-T_1)}\eta_\lambda(z-b )\quad t>T_1,\ z\in\R. 
$$
Thus 
$$
|\phi(z) -w(t,z)|_\lambda \leq (q_2e^{\gamma T_1})e^{-\gamma t }, \ t \geq T_1,
$$
that proves the second statement of the theorem. 
\qed
\end{proof}
\section{Stability lemma and  invariance of the leading asymptotic term}\label{Sec4}
In this section, we are presenting two additional results.  First we demonstrate  a 
quite general local stability lemma which will be used later in the proof of Theorem \ref{MER}. 
Below we take  $q_*, q^*, \delta^*, \gamma^*,  b >0$ as at the beginning of Section \ref{sub3}. 
\begin{lemma}\label{uls}  Assume that 
$c>c_*$ and write, for short, 
$\eta_{1}(z)=\min\{1,e^{\lambda_1(c) z}\}$ instead of $\eta_{\lambda_1}(z)$.  Then   $$
w_{\pm}(t,z):= \phi(z\pm \epsilon_\pm(t)) \pm qe^{-\gamma t}\eta_1(z), \ q \in (0, \min\{q^*,q_*\}],$$
are  super- and sub-solutions for   appropriately chosen functions  
\begin{eqnarray}\hspace{-.0cm}
\epsilon_+(t):= \frac{\alpha q}{\gamma}(e^{\gamma h}- e^{-\gamma
t})>0,  \quad   \epsilon_-(t): =-\frac{\alpha q}{\gamma}e^{-\gamma t}<0, \quad  t > -h. \nonumber
\end{eqnarray}
The parameters
$\alpha >0$  and $\gamma \in (0, \gamma^*)$ are fixed later in the proof and depend only on $g, \phi,
c, h$. 
\end{lemma}
\begin{proof}  Set $z_*=0$ and observe that the smoothness conditions of Definition \ref{def1} and  the second inequality in
(\ref{sso})  are satisfied in view of 
$$\pm\left(\frac{\partial w_\pm(t,0+)}{\partial z}- \frac{\partial
w_\pm(t,0-)}{\partial z}\right)= -  q\lambda_1(c) e^{-\gamma t} <0.$$ 
In order to establish 
the first inequality of (\ref{sso}),  we proceed with the following direct calculation: 
$$
\pm{\mathcal N}w_\pm(t,z):= \epsilon_\pm'(t)\phi'(z\pm\epsilon_\pm(t))- \gamma
q e^{-\gamma t}\eta_1(z)\mp \phi''(z\pm\epsilon_\pm(t))- q e^{-\gamma
t}\eta_1''(z)$$
$$\pm c\phi'(z\pm\epsilon_\pm(t))
+cq e^{-\gamma t}\eta_1'(z)\pm \phi(z\pm \epsilon_\pm(t))+ q e^{-\gamma
t}\eta_1(z)\mp g(w_{\pm}(t-h,z-ch)) \geq$$
$$  \alpha  q e^{-\gamma
t}\phi'(z\pm\epsilon_\pm(t))
 - \gamma q e^{-\gamma t}\eta_1(z)+ cq e^{-\gamma t}\eta_1'(z) +q e^{-\gamma t}\eta_1(z)-
q e^{-\gamma t}\eta_1''(z)$$
$$
\pm \left(g(\phi(z-ch\pm \epsilon_\pm(t)))-g(\phi(z-ch\pm\epsilon_\pm(t)) \pm
qe^{-\gamma (t-h)}\eta_1(z-ch))\right), \ z \not=0. $$
Here we are using the fact that  $g,
\phi, \epsilon_\pm$ are  increasing functions. 

From now on, we fix positive number  $$\gamma < \min\{\gamma^*, (g'(0)-1)e^{-\lambda_1ch}\min\{1,\lambda_1^{-1}\}\}$$ and
$d$, $\alpha$ defined by
\begin{equation}\label{al}
d:= \inf_{z\leq b} \phi'(z)/\eta_1(z)>0\quad\mbox{and}  \quad
\alpha:= d^{-1}e^{\gamma h}L_g.
\end{equation}
Note  that $\alpha$, $d$,  $\gamma$ depend only on $g,
\phi, c, h$.

We claim that  $\pm {\mathcal
N}w_\pm(t,z) \geq 0$ for all $z\not=0, \ t \geq 0$ and $q \in (0, \min\{q^*,q_*\}]$.

Indeed, suppose first that  $z\pm\epsilon_\pm(t) \leq b$. Then   
we find that 
$$0\geq \pm\left(g(\phi(z-ch\pm\epsilon_\pm(t)))-g(\phi(z-ch\pm \epsilon_\pm(t))
\pm qe^{-\gamma (t-h)}\eta_1(z-ch))\right) \geq $$ $$
 -L_g qe^{-\gamma
(t-h)}\eta_1(z-ch), \quad  \pm {\mathcal N}w_\pm (t,z) \geq $$ $$q e^{-\gamma
t}\left\{\eta_1(z\pm\epsilon_\pm(t))d\alpha+(
[1- \gamma] \eta_1(z)+ c\eta_1'(z)
-\eta_1''(z)- e^{\gamma h}L_g\eta_1(z-ch))\right\} $$
$$\geq  q e^{-\gamma
t}\left(\eta_1(z\pm\epsilon_\pm(t))d\alpha- e^{\gamma
h}L_g\eta_1(z-ch)\right) > 0.
$$
Similarly, if $z\pm\epsilon_\pm(t) \geq b,$ then invoking  (\ref{gg1})  and (\ref{gg}) we obtain,  for all $t \geq 0$, that 
$$0\geq \pm\left(g(\phi(z-ch\pm\epsilon_\pm(t)))-g(\phi(z-ch\pm \epsilon_\pm(t))\right)
\pm qe^{-\gamma (t-h)}\eta_1(z-ch)) \geq $$ $$ -qe^{-\gamma t}\eta_1(z-ch)(1-2\gamma),\quad \pm{\mathcal N}w_\pm(t,z) \geq $$ 
$$q e^{-\gamma t}\left( [1- \gamma]
\eta_1(z)+ c\eta_1'(z)
-\eta_1''(z)- (1-2\gamma)\eta_1(z-ch)\right) \geq $$
$$ q e^{-\gamma t} \left\{
\begin{array}{ll} e^{\lambda_1 z}
     [1- \gamma+ c\lambda_1
-\lambda_1^2- e^{-\lambda_1 ch}(1-2\gamma)], & {z< 0} \\
     \gamma, & {z > 0}
\end{array}%
\right\} >0.
$$
The proof of Lemma \ref{uls} is completed. \qed
\end{proof}
\begin{corollary}\label{In} Let $c>c_*$ and $\gamma >0$ be as in Lemma \ref{uls} and $\alpha$ be as in (\ref{al}).  Then there exists positive
number  $C=C(g,\phi)$ such that for each 
non-negative initial function $w_0$ satisfying 
\begin{equation*}\label{mlem} \phi(z)-q_- \eta_1(z) \leq w_0(s,z)\leq \phi(z)+q_+ \eta_1(z), \quad z\in\mathbb{R},\quad s\in [-h,0],
\end{equation*}
for some $0< q _\pm  \leq \varsigma_0:= \min\{\gamma, \min\{q_*,q^*\}\exp\left(- \lambda_1\alpha e^{\gamma h}\right)\}$,  
it holds 
\begin{equation}\label{c7}\phi(z-Cq_-)-Cq_-e^{-\gamma t}
\eta_1(z)  \leq w(t,z)\leq \phi(z+Cq_+)+Cq_+e^{-\gamma t}
\eta_1(z),
\end{equation} for all $
  z\in\mathbb{R}, t\geq -h$. 
\end{corollary}
\begin{proof}  The right hand side inequality  in (\ref{c7}) is a direct consequence of Lemmas \ref{cl} and \ref{uls} in view  of the 
estimations 
$$
w_0(s,z) \leq \phi(z)+q_+\eta_1(z)\leq \phi(z+\epsilon_+(s))+q_+e^{-\gamma s}\eta_1(z), \
(z,s) \in \Pi_0.$$
Since $\epsilon_+(t)$ increases on $\R$, this proves this part of  inequality (\ref{c7})  with $C=C_1:= \alpha e^{\gamma h}/\gamma$.

In order to prove the left hand side inequality  in (\ref{c7}), observe that 
$$
w_0(s,z-\epsilon_-(-h))\geq  \phi(z-\epsilon_-(s))-q_-e^{-\lambda_1\epsilon_-(-h)} e^{-\gamma s}\eta_1(z) \geq 
$$
$$
 \phi(z-\epsilon_-(s))- \min\{q_*,q^*\} e^{-\gamma s}\eta_1(z), \quad
(z,s) \in \Pi_0.$$
This implies that, for all $t \geq -h$, $z \in \R$, it holds  
$$
w(t,z)\geq  \phi(z - \epsilon_+(t))-q_-e^{-\lambda_1\epsilon_-(-h)} e^{-\gamma t}\eta_1(z) \geq$$
 $$\phi(z - C_1q_-)-C_2q_- e^{-\gamma t}\eta_1(z), \ C_2:= \exp\left( \lambda_1\alpha e^{\gamma h}\right).$$
 Setting $C= \max\{C_1,C_2\}$, we complete the proof of Corollary \ref{In}. 
 \qed
 \end{proof}
\begin{corollary}\label{coco}
For every $\epsilon >0$ there exists $\varsigma(\epsilon) >0$ such
that $$|\phi(\cdot ) - w(s,\cdot)|_{\lambda_1} < \varsigma(\epsilon),\quad s \in
[-h,0],$$ implies $|\phi(\cdot ) - w(t,\cdot)|_{\lambda_1}
< \epsilon$ for all $t \geq 0$.
\end{corollary}
\begin{proof} It suffices to take 
$$
\varsigma(\epsilon) = \min\left\{\varsigma_0, \frac{\epsilon}{C(1+e^{\lambda_1C\varsigma_0}\sup_{z \in \R}\left[\phi'(z)/\eta_1(z)\right])} \right\},
$$
where $C= \max\{C_1,C_2\}$ was defined in the proof of Corollary \ref{In} and to apply  Corollary \ref{In}.  \qed
\end{proof}
The second main result of this section assures the invariance of the main asymptotic term at $-\infty$ of solutions with `good' initial data.  It sheds some new light on the conclusions of Corollary \ref{zd}A.
\begin{lemma}\label{Inv}
Suppose that the birth function $g$ is bounded and that there exists $g'(0)>1$. If  the initial fragment $u(s,z)$ of a bounded solution $u(t,z)$ to equation 
(\ref{e1}) is such that, for some positive eigenvalue $\lambda_j(c), \ j =1,2$, it holds   that $u(s,x-cs)e^{-\lambda_j(c)x} \to 1, \ x \to -\infty,$  for each
$s\in[-h,0]$. 
Then also it holds that $u(t,x-ct)e^{-\lambda_j(c)x} \to 1, \ x \to -\infty,$ for each $t\geq 0$. 
\end{lemma}
 \begin{proof}Due to a step by step argument, it is sufficient to consider the situations when $t \in [0,h]$. 
 Set  $U(t,x):=e^{t}u(t,x)$, then $U(s,x-cs)e^{-\lambda_j(c)x} \to e^s, \ x \to -\infty,$ and 
$$
U_t (t,x)=U_{xx}(t,x)+e^{t}g(e^{-t+h}U(t-h,x)), \quad \ t>0,\ x\in\R.
$$
Hence, by Duhamel's formula (see e.g. \cite[Theorem 12, p. 25]{AF}), 
$$
U(t,x) = \Gamma(t, \cdot) *U(0,\cdot) + \int_0^t \Gamma(t-s, \cdot) *e^{s}g(e^{-s+h}U(s-h,\cdot))ds, 
$$
$$
\mbox{where }\qquad \Gamma(t,x) = \frac{1}{2\sqrt{\pi t}} e^{-x^2/4t}, \ t >0,\ x \in \R, 
$$
is the fundamental solution and $\Gamma(t, \cdot) *U(s,\cdot)$ denotes the convolution on $\R$ with respect to the missing space variable. 

By Lebesgue's dominated convergence theorem, for each $s \in [-h,0], t > 0$, 
$$
\lim_{x\to-\infty}e^{-\lambda_j x}\Gamma(t, \cdot) *
U(s,\cdot)=$$
$$\frac{1}{2\sqrt{t\pi}}\int_{\R}e^{-\frac{1}{4t}[(y+2t\lambda_j)^{2}-4t^{2}\lambda_j^{2}]}\lim_{x\to-\infty}
e^{-\lambda_j(x-y)}U(s, x-y)dy
=e^{\lambda_j^{2}t+\lambda_jcs+s}.
$$
Consequently, for $t \in (0,h]$, we have that 
\begin{eqnarray*}
&&\hspace{-3mm} \lim_{x\to-\infty}e^{-\lambda_j x}U(t,x)=e^{\lambda_j^{2}t}+\int_{0}^{t}\lim_{x\to-\infty}e^{-\lambda_j
x}\Gamma(t-s,\cdot)* g(e^{-s+h}w(s-h,\cdot))e^{s}ds\\
&=& e^{\lambda^{2}t}+g'(0)e^{-\lambda
ch}e^{\lambda^{2}t}\int_{0}^{t}e^{(-\lambda^{2}+\lambda c+1)s}ds
=e^{(1+\lambda c)t}.
\end{eqnarray*}
Finally,  we obtain the relation 
$
\lim_{x\to-\infty}e^{-\lambda_j
x}u(t,x)=e^{\lambda_j ct}
$ for each $t \in (0,h]$ which completes the proof of the lemma. 
\qed
\end{proof}
\begin{remark} An obvious modification of the  proof of Lemma \ref{Inv} yields the following assertion:
{\it Assume that the birth function $g:\R_+\to \R_+,\ g(0)=0,$ is bounded and  Lipschitz continuous. Suppose also that 
the initial fragments $u_k(s,z),\ k=1,2,$ of bounded solutions $u_k(t,z)$ to equation 
(\ref{e1}) satisfy,  for some positive $\mu$,  the relation 
$$(u_1-u_2)(s,x-cs)e^{-\mu x} \to 0, \ x \to -\infty,\ s\in[-h,0].$$ 
Then $(u_1-u_2)(t,x-ct)e^{-\mu x} \to 0, \ x \to -\infty,$ for each $t\geq 0$. }

This result provides a short and elementary  justification for  a delicate aspect of getting {\it a priori} estimates for a weighted energy  method  developed by Mei {\it et al.} \cite{MLLS,MLLS2,MeiI,MeiF}.  Indeed, an important  initial fragment of the derivation of these estimates includes elimination of the boundary term $$(u-\phi)(t,x-ct)e^{-\mu x}|_{x=-\infty} = (w(t,z)-\phi(z))e^{-\mu z}e^{\mu ct}|_{z=-\infty} =0.$$ For instance, see \cite[p. 855]{LvW},  \cite[formulas (3.9)-(3.11)]{MeiI} or \cite[p. 1067]{LLLM}. 
\end{remark}
\section{Proof of Theorem \ref{MER}} \label{sub4}
We start by establishing the following result.
\begin{lemma}\label{Lem2} 
Assume that the initial function $w_{0}(s,z)\geq 0$ is uniformly bounded on the strip  $[-h,0]\times \R$ 
(say, by some $K>0$) and satisfies  the hypothesis $(IC2)$ and, for some $c> c_*$, it holds 
$$
\lim_{z\to -\infty} w_0(s,z)
e^{-\lambda_1(c)z} =1
$$
uniformly on $s\in[-h,0]$.
Then for each $\varsigma >0$ there exists $L\in \R$ and $\psi \in C^2(\R)$  such that 
$\psi(z) = (1+\varsigma +o(1))e^{\lambda_1(c)z},$ $ z \to -\infty,$  and $\psi'(z) >0$ for $z \in \R$,  $\psi(L)=K,$ $\psi(+\infty) = +\infty$, $w_0(s,z) < \psi(z), \ z \leq L, s \in [-h,0]$, and 
\begin{equation*}\label{psih}
\psi''(z)-c\psi'(z)-\psi(z)+g(\psi(z-ch))< 0, \ z \leq L. 
\end{equation*}
\end{lemma}
\begin{proof} Since $c> c_*$, the linearisation of  equation (\ref{EP}) about $0$ has exactly two 
real simple eigenvalues $\lambda_1(c) < \lambda_2(c)$. In particular,  the linearised equation has a 
positive solution $(\phi(t), \phi'(t)) = (1,\lambda_2(c)))e^{\lambda_2(c)t}$. Moreover, the eigenvalue $\lambda_2 = \lambda_2(c)$  is dominant (i.e. $\Re \lambda_j (c) < \Re \lambda_2$ for all other eigenvalues $\lambda_j (c), j \not=2$).  As a consequence, equation (\ref{EP}) has a  solution $\psi_2(t)$ with the following asymptotic behaviour at $-\infty$: 
$$
 (\psi_2(t), \psi_2'(t)) = (1,\lambda_2)e^{\lambda_2t} + O(e^{(\lambda_2+\epsilon)t}), \ t \to -\infty, \ \epsilon >0
$$ 
(see e.g. \cite[Theorem 2.1]{FTnon} for more detail).  

In this way, there exists a maximal open non-empty interval $(0,T),$ $T \in \R \cup \{+\infty\},$ such that $\psi_2(t) >0,\ \psi_2'(t) >0$   for all $t \in (0,T)$. 

\underline{\sf We claim} that $\psi_2(T) >\kappa$ and $T=+\infty$.  First, it should be noted that  $\psi_2(T) \not=\kappa$ since otherwise we obtain a)  if $T$ is finite then   $\psi_2(T)=\kappa > g(\psi_2(T-ch)),$ $\psi'_2(T)=0,$  $\psi''_2(T)\leq 0,$ contradicting  (\ref{EP}); b) if $T=+\infty$ then $\psi_2(t)$ is  a monotone heteroclinic connection between $0$ and $\kappa$, different from $\psi_1$.  Here $\psi_1(t)$ denotes the unique monotone wavefront to  (\ref{EP}) normalised by the condition 
$\psi_1(t)e^{-\lambda_1t} = 1 +o (1),$ $ t \to -\infty$.  This contradicts  the uniqueness of the wavefront $\psi_1$ established in \cite{TPT}.  Next, suppose that $\psi_2(T) <\kappa$ and consider the difference $\theta_a(t) = \psi_1(t)-\psi_2(t+a), \ t \in \R,$ for some fixed $a\in \R$. Since $\psi_1$ is a strictly monotone heteroclinic connection between $0$ and $\kappa$, there exists a unique $S\in \R$ such that $\psi_1(S)=\psi_2(T)$. Now, taking into account  the inequality $\lambda_1 <\lambda _2$, we obtain that, for each fixed $a$, the function 
$\theta_a(t)$ is positive  in some maximal interval $(-\infty, \sigma(a))$.  If we choose $b= T-S$ then 
$\theta_b(S)=0,$ $\theta'_b(S) >0$ and therefore $\sigma(b)=\sigma(T-S)<S$, $\theta_b(\sigma(b))=0$.  On the other hand,  $\theta_{a_1}(t) >0, \ t \in [\sigma(b),S]$,  for some large negative $a_1\leq b$. Note also that $\theta_a(t) >\theta_b(t) > 0, \ t \leq \sigma(b)$ if $a <b$.  In consequence, there exists  $d \in (a_1,b]$ such that $\theta_d(\sigma(d))=\theta_d'(\sigma(d))=0\leq \theta_d''(\sigma(d))$ and $\theta_d(\sigma(d)-ch) = \psi_1(\sigma(d)-ch)-\psi_2(d+\sigma(d)-ch) >0$.  However, this yields the following contradiction: 
$$
0= \theta_d''(\sigma(d))-c\theta_d'(\sigma(d))-\theta_d(\sigma(d))+g(\psi_1(\sigma(d)-ch)) - g(\psi_2(d+\sigma(d)-ch)) >0
$$
because  $g$ is strictly increasing. 

Finally, if $T<+\infty$ and $\psi_2(T) >\kappa$, then  $g(\psi_2(T-ch))< g(\psi_2(T))< \psi_2(T)$. Since, in addition, $\psi''_2(T)\leq \psi'_2(0)=0$, the following contradiction
$$0=\psi_2''(T)-c\psi_2'(T)-\psi_2(T)+g(\psi_2(T-ch))< 0$$
proves the above \underline{\sf claim}. 

Next, we consider, for $\epsilon \in [0,1]$, $\theta, \delta_0$ as in {\rm \bf(H)} and $\mu \in (\lambda_1(c), \lambda_2(c))$, $\mu < (1+\theta)\lambda_1(c)$, the function
$$
\psi(t,\epsilon) = \psi_2(t) + \epsilon (e^{\lambda_1 t} + e^{\mu t}).
$$
It is clear that $\psi(t,\epsilon) \leq Ce^{\lambda_1t}, \ t \leq 0,$ for some $C>1$ which does not depend on $\epsilon \in [0,1]$. 

With  $\chi_0(z) = z^2 -cz -1 +g'(0)e^{-zch}$, we have that $\chi_0(\mu) <0$ and  
$$
\mathcal{D}{\psi}:= \psi''(t,\epsilon)-c\psi'(t,\epsilon)-\psi(t,\epsilon)+g(\psi(t-ch, \epsilon)) = $$
$$ {\epsilon} \chi_0(\mu)e^{\mu t} + g(\psi(t-ch, \epsilon))- g(\psi(t-ch, 0)) -
g'(0)\epsilon (e^{\lambda_1 (t-ch)} + e^{\mu (t-ch)t}).
$$
Let $T_0 <0$ be such that $\psi(t-ch, \epsilon) \leq \delta_0:= \psi(T_0, 1)$ for all $t \leq T_0,$ \ $\epsilon \in [0,1]$. Then, for some $P(t,\epsilon) \in [\psi(t-ch, 0), \psi(t-ch, \epsilon)]$, it holds that 
$$
 |g(\psi(t-ch, \epsilon))- g(\psi(t-ch, 0)) -
g'(0)\epsilon (e^{\lambda_1 (t-ch)} + e^{\mu (t-ch)t})|= 
$$
$$
 |g'(P(t,\epsilon))- 
g'(0)|\epsilon(e^{\lambda_1 (t-ch)} + e^{\mu (t-ch)t})\leq $$
$$\epsilon(\psi(t-ch, \epsilon))^\theta|(e^{\lambda_1 (t-ch)} + e^{\mu (t-ch)t})\leq 
2C{\epsilon}e^{(1+\theta) \lambda_1 t}, \ t \leq T_0.
$$
Thus, for a sufficiently large negative $T_1 <T_0$,  
$$
\mathcal{D}{\psi} \leq {\epsilon} e^{\mu t}(\chi(\mu) + 2Ce^{[(1+\theta) \lambda_1-\mu] t}) <0
$$
for all $\epsilon \in (0,1],$\  $t \leq T_1$.  As a consequence, if we define $\psi_\epsilon(t)$ by  
\begin{equation*}
\psi_\epsilon(t):=\left\{\begin{array}{ll}\psi(t,\epsilon),&0\leq t\leq T_1, \\ y(t,\epsilon),& T_1 \leq t,\end{array}\right.
\end{equation*}
where $y=y(t,\epsilon),\ t \geq T_1,$ solves the initial value problem 
$y(s,\epsilon) = \psi(s,\epsilon),$ \ $s \in [T_1-ch,T_1],$ $y'(T_1,\epsilon) = \psi'(T_1,\epsilon)$ for the equation 
$$y''(z)-cy'(z)-y(z)+g(y(z-ch))= \mathcal{D}{\psi}(T_1,\epsilon)<0 ,$$
then $\psi_\epsilon \in C^2(\R)$ and $\mathcal{D}{\psi_\epsilon}(t) <0$. Define $T_K$ as the unique solution of the equation $\psi_2(T_K)=K$, then due to the smooth dependence of the initial function and  $\mathcal{D}{\psi}(T_1,\epsilon), \mathcal{D}{\psi}(T_1,0)=0,$ on the parameter $\epsilon$, 
$$
(y(t,\epsilon), y'(t,\epsilon)) \to (\psi_2(t), \psi'_2(t)), \quad \epsilon \to 0+, 
$$
uniformly for $t \in [T_0, T_K]$.   

Finally, due to the assumptions imposed on $w_0$, there exists $T_2<T_1$ such that
$$
w_0(t,s) \leq (1+\varsigma)e^{\lambda_1(c)z}<\psi_2(T_1), \ t \leq T_2, \ s \in [-h,0].
$$  
For $\epsilon \in (0,1]$,  set $p_\epsilon = \lambda_1^{-1}(c)\ln [(1+\varsigma)/\epsilon]$ and $\tilde \psi(t):= \psi_\epsilon(t+p_\epsilon)$.  Obviously,  $\tilde \psi(t) > (1+\varsigma)e^{\lambda_1(c)z},$ $t \leq T_1-p_\epsilon,$ $\tilde \psi(t)> \psi_2(T_1)$, $t \in [T_1-p_\epsilon, T_K-p_\epsilon]$ and $\tilde \psi(t) =(1+\varsigma +o(1))e^{\lambda_1(c)z},$ $t \to +\infty$.  Since 
 $\tilde \psi(T_K-p_\epsilon) = y(T_K,\epsilon) >K$, we obtain that 
 $$
 w_0(s,z)  \leq \tilde \psi(z), \quad s \in [-h,0].
 $$
whenever $T_K< T_2+p_\epsilon$. 
\qed
\end{proof}
Next, for the solution $w(t,z)$ of the initial value problem $w(s,z) = w_0(s,z),$ $(s,z) \in [-h,0]\times \R$, we define its  $\omega-$limit set by
$$
\Omega(w_0)=\{w_{*}\in C^{1,2}([-h,0]\times\R): \mbox{there exists  some\ }  t_{k}\to +\infty \ \mbox{such that \ } $$
$$\lim_{k\rightarrow\infty}w(t_k +s,z)=w_{*}(s,z)\
\mbox{uniformly on compact subsets of \ }   [-h,0]\times \R
 \}. 
$$
Note that the set $\Omega(w_0)$ is non-empty, compact and invariant with respect to the flow generated by  equation (\ref{E1}),  e.g. see  \cite[Lemma 2.8]{STR}.   
\begin{theorem}\label{Te4}
Assume that the initial function $w_{0}(s,z)\geq 0$ satisfies  the hypotheses $(IC1)$, $(IC2)$ and that, for some
$A >0$ and $c> c_*$, it holds 
$$
\lim_{z\to -\infty} w_0(s,z)
e^{-\lambda_1(c)z} =A
$$
uniformly on $s\in[-h,0]$. Choose a shifted copy of the wavefront profile $\phi$ normalised by the boundary condition 
$\lim_{z\rightarrow-\infty}e^{-\lambda_1(c)z}\phi(z)=1$. Then  
$$
\lim_{t\rightarrow\infty}|\phi(\cdot+a)-w(t,\cdot)|_{\lambda_1}=0,
$$
where  $w(t,z)$ solves   the initial value problem 
$w(s,z) = w_0(s,z), (s,z) \in \Pi_0,$ for (\ref{E1})  and $a=(\lambda_1(c))^{-1}\ln A$.
\end{theorem}
\begin{proof} Without loss of generality, we may assume that $A=1$ (otherwise we can take a shifted copy of $w_0$).  Fix an arbitrary  $\varsigma >0$ and let  
$L, K$ and $\psi$ satisfy all the conclusions of Lemma \ref{Lem2}. Then we have that  $w_0(s,z) \leq \psi_+(z)$, $(s,z) \in [-h,0]\times \R$, where 
\begin{equation*}\label{dg}
\psi_+(z):=\left\{\begin{array}{ll}\psi(z),&0\leq z\leq L, \\ K,& L \leq z.  \end{array}\right.
\end{equation*}
Since $K>g(K) \geq g(\psi_+(z-ch))$ and $\psi_+'(L-) > 0 = \psi_+'(L+)$, we conclude that  $\psi_+(z)$ is a super-solution for equation (\ref{E1}).  In view of Lemma \ref{cl}, we also find that 
$$
w(t,z)\leq\psi_+(z), \quad (t,z)\in\R_{+}\times \R.$$

 On the other hand, it is easy to see (e.g., cf. \cite[p. 478]{U}) that there exists a strictly increasing  $C^1$-function $\hat{g}: \R_+ \to \R_+$ satisfying the hypothesis $(\mathbf{H})$ and such that  $g'(0)=\hat{g}'(0)\geq \hat{g}'(x), $ \
 $g(x)\geq\hat{g}(x)$ for all $x\in[0,\kappa]$.  Let $\hat{w}(t,z),$ $t >0, z \in \R,$ solve the initial value problem 
\begin{equation}\label{hat}
w_{t}(t,z)=w_{zz}(t,z)-cw_{z}(t,z)-w(t,z)+\hat g(w(t-h,z-ch)),
\end{equation}
$$ w(s,z) = w_0(s,z), \ s\in [-h,0], \
z\in\R, 
$$
then clearly ${w}(t,z)$ is a super-solution for  (\ref{hat}) and therefore Lemma \ref{cl} implies that 
$\hat w(t,z)\leq {w}(t,z)$ for all $(t,z)\in\R_{+}\times\R$. Furthermore, 
 Theorem \ref{Te3}A assures  that
$\lim_{t \to +\infty}|\hat{w}(t,\cdot)-\hat{\phi}(\cdot)|_{\lambda_1}=
0$ for the wavefront $\hat \phi$ of equation (\ref{hat})  normalised  by $\lim_{z\rightarrow-\infty}e^{-\lambda_1(c)z}\hat \phi(z)=1$. 

Next, let ${w}_u(t,z),$ $t >0, z \in \R,$ denote the solution of  the initial value problem 
$ w_u(s,z) = \psi_+(z), \ s\in [-h,0], \
z\in\R,$ for equation (\ref{E1}).  Then  Corollary \ref{cod1} implies that 
\begin{equation}\label{u}
\hat w(t,z)\leq {w}(t,z)\leq w_u(t,z) \leq \psi_+(z), \ (t,z)\in\R_{+}\times\R. 
\end{equation}
Therefore it holds,  for some $a_1 \in [0, \lambda^{-1}(c)\ln (1+\varsigma)]$ and for all $w_l \in \Omega(w_0)$, that 
\begin{equation}\label{su}
\hat \phi(z) \leq w_l(s,z) \leq \phi(z+a_1), \ z \in \R,\ s \in [-h,0],
\end{equation}
where 
$$1=\lim_{z \to -\infty} \hat \phi(z) e^{-\lambda_1 z} \leq \lim_{z \to -\infty} \phi(z+a_1) e^{-\lambda_1 z} \leq \lim_{z \to -\infty} \psi_+(z) e^{-\lambda_1 z}= 1 +\varsigma.$$
 
Next, since $\hat \phi(z)$ is a sub-solution for equation (\ref{E1}), we find analogously that, for some $a_0\in [0,a_1]$ and for all $w_{ll} \in \Omega(w_l)\subset \Omega(w_0)$,  
$$\phi(z+a_0) \leq w_{ll}(s,z) \leq \phi(z+a_1), \ z \in \R,\ s \in [-h,0],$$
where 
$$1\leq \lim_{z \to -\infty} \phi(z+a_0) e^{-\lambda_1 z} \leq \lim_{z \to -\infty} \phi(z+a_1) e^{-\lambda_1 z} \leq  1 +\varsigma.$$
Since the latter relation holds for every $\varsigma >0$, we conclude that  actually $a_0=0$ and $\{\phi(\cdot)\} = \Omega(w_l)\subset \Omega(w_0)$.  Furthermore, as  a consequence of (\ref{su}), 
$\lim_{z \to -\infty}e^{-\lambda z} w_l(s,z) =1$ uniformly in $s \in [-h,0]$. 

Hence, for each $\varsigma >0$ there are  $Z_1(\varsigma), \ T_\varsigma >0$ such that, for all  $t \geq T_\varsigma,$ $z \leq Z_1(\varsigma)$, it holds 
\begin{equation}\label{dig}
\begin{array}{ll}& -2\varsigma \leq e^{-\lambda_1 z}(\hat{w}(t,z)-\hat{\phi}(z)) - e^{-\lambda_1 z}({\phi}(z)-\hat{\phi}(z)) \leq  \\ & e^{-\lambda_1 z}({w}(t,z)-{\phi}(z)) \leq e^{-\lambda_1 z}(\psi_+(z)-{\phi}(z)) < 2\varsigma.   \end{array}
\end{equation}
 In addition, $\{\phi(\cdot )\} \in \Omega(w_0)$ implies that there exits a sequence $t_n\to +\infty$ that $w(t_n+s,z) \to \phi(z)$ on compact subsets of $\Pi_0$. This fact, together with (\ref{u}) and (\ref{dig}), implies that 
 $$
\sup_{s \in [-h,0]} |\phi(\cdot)-w(t_n+s,\cdot)|_{\lambda_1}\leq 2\varsigma
 $$
 for all sufficiently large $n$. Finally,  an application  of Corollary \ref{coco} completes the proof. \qed
\end{proof}
{Below, we use Theorem \ref{MER} in order to analyse behavior of solutions whose initial data satisfy  the hypotheses $(IC1)$, $(IC2)$ and $(\ref{sps})$}.
\begin{proof}{\it of Corollary  \ref{Cor1}}:

\noindent\underline{Case I: $\lambda > \lambda_*$}.  The statement of the corollary is an immediate consequence of \cite[Theorem 1.4]{STR}. 

\noindent\underline{Case II: $\lambda < \lambda_*$}. 
Clearly, 
$\lambda=\lambda_{1}(c(\lambda))$.
Set $A_{-}=\min_{s\in[-h,0]} A(s)e^{-\mu s}$. Then for each 
 $A_1 <A_-$,  the initial datum 
$$
w_1(s,x):=\min\{A_1e^{\lambda(x+cs)},w_0
(s,x)\}
$$
meets all the conditions of Theorem \ref{MER}. Consequently,   for each $\delta>0$ there
exists $T_\delta >0$ such that solution $u_1(t,x)$ of the initial value problem $u_1(s,x)= w_1(s,x), \ (s,x) \in \Pi_0,$ to equation (\ref{e1}) satisfies 
\begin{eqnarray*}
\phi(x+ct+a_1)-\delta\eta_{\lambda}(x+ct)\leq u_1(t,x), \quad \mbox{for all } \
x\in\R,\  t>T_{\delta} 
\end{eqnarray*}
with $a_{1}=\lambda^{-1}\ln\, A_1$. Now, the functions 
$\phi$ and $\eta_\lambda$ are equivalent at $-\infty$ so that, to each given $\epsilon>0$ we can find  $A_1$ close to $A_-$ and $\delta>0$ close to $0$ such that 
$$
(1-\epsilon)\phi(x+ct+a_{-}) \leq
u_{1}(t,x) \leq u(t,x), \quad x\in\R,\ t>T_{\delta}.
$$
The upper estimation can be established  in a  similar way by comparing $u(t,x)$ with solution $u_2(t,x)$ of  (\ref{e1}) satisfying the initial condition 
$$w_2(s,x)=\max\{A_2e^{\lambda(x+cs)},w_{0}(s,x)\}, \quad (s,x) \in \Pi_0,$$ 
with $A_2 > A_+=\max_{s\in[-h,0]}A(s)e^{-\mu s}$.

\noindent\underline{Case III: $\lambda = \lambda_*$}.  
In order to establish inequalities (\ref{lc}), we proceed in the same manner as in Case II by  taking 
the initial functions 
$$
\tilde w_1(s,x):=\min\{A_1e^{M(x+cs)},w_0
(s,x)\}, \ \tilde w_2(s,x)=\max\{A_2e^{\nu(x+cs)},w_{0}(s,x)\},
$$
where $\nu < \lambda_* < M < \lambda_2(c_*)$, instead of $w_1(s,x)$ and $w_2(s,x)$. In addition, while proving the left side inequality of (\ref{lc}), 
we have also to  use  \cite[Theorem 1.4]{STR} instead of Theorem \ref{MER} (cf. Case I above). 

Now, inequalities  (\ref{lc})  also imply that the only wavefront to which  
$u(t,x)$ can converge (as $t \to +\infty$) is some translation $\phi_*(x+c_*t +b)$ of the critical wavefront $\phi_*(x+c_*t)$.  However, this is not possible in view of the following argument. 
Take some  $A_1 <A_-$ and some strictly increasing $\hat{g}\leq g$ satisfying
$(\textbf{H})$ with $L_{\hat{g}}=g'(0)$. Set 
$$
w_{*}(s,x)=\min\{A_1e^{\lambda_{*}(x+cs)},w_{0}(s,x) \}. 
$$
Then by the comparison principle,  solution $w_*(t,x)$ of the initial value problem
$$
w_{t}(t,x)=w_{zz}(t,x)-w(t,x)+\hat g(w(t-h,x)), \ w(s,x) =w_*(s,x), \ (s,x) \in \Pi_0, 
$$
satisfies $w_*(t,x) \leq u(t,x)$ for all $t \geq 0$, $x \in \R$. 
On the other hand, by 
invoking  Theorem \ref{Thm2}, we find that $w_*(t,x)$ converges uniformly to some wavefront
$\hat{\phi}_{*}(x+c_*t)$ of the modified equation.  Keeping $z= x+c_*t$ fixed and passing to the limit in $w_*(t,x) \leq u(t,x)$ (as $t \to +\infty$) for each fixed $z$, we find that $\hat{\phi}_{*}(z) \leq \phi_*(z +b)$ for all $z \in \R$.  However, this is not possible since 
$\phi_*(z+b)$ decays at $-\infty$ faster than $\hat{\phi}_{*}(z)$. 

Finally, in order to prove inequality (\ref{lc2}), it suffices to consider 
the initial function
$$
 \tilde w_3(s,x)=\max\{-xe^{\lambda_*(x+cs)},w_{0}(s,x)\},  \ (s,x) \in \Pi_0, 
$$
instead of  $w_2(s,x)$. Then we proceed can similarly to the proof of  inequalities  (\ref{lc})   by applying Theorem \ref{Thm2}A.
\qed
\end{proof}
\section{Proof of Theorem \ref{Thm3} and Corollary \ref{co3}} \label{sub6} Let the triple $(c,\lambda_c, \gamma)\in [c_\#, +\infty)\times  [\lambda_1(c),  \lambda_2(c))\times \R_+$ be as in Lemma \ref{Sttg} (i.e. 
 $\lambda_c= \lambda_1, \gamma =0$ if $c=c_\#$ and $\gamma >0, \lambda_c \in (\lambda_1(c), \lambda_2(c))$ if $c > c_\#$).   
Theorem \ref{Thm3} and Corollary \ref{co3}  follow from the next three assertions.  
\begin{lemma} \label{L1}Assume (\textbf{UM}) and let the initial function $w_0$ satisfy  (IC1).  Consider  $c \geq c_\#$ and let $\phi(z)$ denote a  positive semi-wavewfront to equation (\ref{EP}).  
Then the inequalities 
\begin{equation*}\label{26}
\phi(z) - qe^{-\gamma s}\xi(z-b,\lambda_c)\leq 
 w_0(s,z) \leq \phi(z) + qe^{-\gamma s}\xi(z-b,\lambda_c), \ (s,z) \in \Pi_0,  
\end{equation*}
(where $q>0, b \in \R$ are some fixed numbers) imply that  the solution $w(t,z)$ of the initial value problem 
$w(s,z) = w_0(s,z), (s,z) \in \Pi_0,$ for (\ref{E1}) satisfies 
\begin{equation}\label{tak}
\phi(z) - qe^{-\gamma t}\xi(z-b,\lambda_c)\leq 
 w(t,z) \leq \phi(z) + qe^{-\gamma t}\xi(z-b,\lambda_c), \  t \geq 0, \ z \in \R.   
\end{equation}
\end{lemma}
\begin{proof}  Set $\delta_\pm(t,z)= \pm(w(t,z)-  (\phi(z) \pm qe^{-\gamma t}\xi(z-b,\lambda_c)))$ and
$$(\mathcal{L}\delta)(t,z):= \delta_{zz}(t,z)- \delta_{t}(t,z)-c\delta_{z}(t,z)-\delta (t,z).$$
Then $$(\mathcal{L}\delta_\pm)(t,z)=
 \mp(g(w(t-h,z-ch))- g(\phi(z-ch))) + q e^{\lambda_c(z-b)}e^{-\gamma t}[-\lambda^2+c\lambda +1-\gamma]. 
$$
Therefore  we obtain, for all $z \in \R$, $t \in (0,h]$, 
$$(\mathcal{L}\delta_\pm)(t,z)\geq 
 q e^{\lambda_c(z-b)}e^{-\gamma t}[-\lambda^2+c\lambda +1-\gamma - g'(0)e^{\gamma h}e^{-ch\lambda_c}]\geq 0. $$
Since, in addition,  $\delta_\pm(0,z) \leq 0$ and $\delta(t,z)$ is exponentially bounded, an application of  the Phragm\`en-Lindel\"of principle yields $\delta_\pm(t,z) \leq 0$ for all $t\in [0,h]$. Finally,  step by step procedure completes the proof of the inequality  $\delta_\pm(t,z) \leq 0$ for all $t \geq 0$. 
\qed
\end{proof}
\begin{lemma} \label{L7} Let all the conditions of Lemma \ref{L1} be satisfied and $c>c_\#$.  Assume, in addition,  that $|g'(u)| <1$ on some interval
$[\kappa-\rho, \kappa + \rho]$, $\rho>0$.  If, for some $b\geq 0$, 
the initial function $w_0$ and the semi-wavefront profile $\phi_c$ satisfy $$
\kappa-\rho/4 \leq w_0(s,z),\phi(z) \leq \kappa + \rho/4 \ \mbox{for all}\ z\geq b-ch, \ s \in [-h,0], 
$$
\begin{equation}\label{tiki}
|w_0(s,z)-\phi(z)| \leq 0.5 \rho e^{\lambda_c(z-b)},\  z\leq b, \ s \in [-h,0], 
\end{equation}
then $\phi$ is actually a wavefront (i.e. $\phi(+\infty)=\kappa$) and  the solution $w(t,z)$ of the initial value problem 
$w(s,z) = w_0(s,z), (s,z) \in \Pi_0,$ for (\ref{E1}) 
satisfies 
\begin{equation}\label{ECO}
|w(t,\cdot)- \phi(\cdot)|_{\lambda_c}  \leq 0.5\rho e^{-\gamma t }, \ t \geq 0, 
\end{equation}
for some $ \gamma >0$. 
\end{lemma}
\begin{proof} Suppose that $\gamma>0$  is sufficiently small to satisfy the inequality 
 $
m_g:= \max\{|g'(u)|: u \in [\kappa-\rho, \kappa + \rho]\} e^{\gamma h} < 1-\gamma
$. 
 Clearly, for all $(s,z) \in \Pi_0$, it holds that
$$
\delta_-(s,z):= \phi(z) -  0.5\rho e^{-\gamma s}\eta_{\lambda_c}(z-b) - w_0(s,z) \leq 0, 
$$
$$
\delta_+(s,z):= w_0(s,z)-\phi(z) -  0.5\rho e^{-\gamma s}\eta_{\lambda_c}(z-b)\leq 0. 
$$
Then Lemma \ref{L1} implies that (\ref{tak}) holds with $q=0.5\rho$. From the proof of Lemma \ref{L1} we know that 
$(\mathcal{L}\delta_\pm)(t,z)\geq 0$ for all $t\in (0,h]$ and $z< b$.  Next, for each $t \in (0,h]$ and $z >b$, we find, by applying the  Lagrange mean value theorem, that 
 $$(\mathcal{L}\delta_\pm)(t,z)=
 \mp(g(w(t-h,z-ch))- g(\phi(z-ch))) + 0.5\rho e^{-\gamma t}[1-\gamma]=
$$
 $$
 \mp(g'(\zeta)(w(t-h,z-ch)- \phi(z-ch))) + 0.5\rho e^{-\gamma t}[1-\gamma]\geq 
0.5\rho e^{-\gamma t}(-m_g+1-\gamma) >0. 
$$
Here $\zeta=\zeta(t,z)$ denotes some point in  $[\kappa-\rho, \kappa + \rho]$. 

Note also that  $\delta_\pm(0,z) \leq 0$, $\delta_\pm(t,z)$ are uniformly bounded on $[0,h]\times \R$ and 
inequality (\ref{di}) is satisfied for $\delta_\pm(t,z)$ with $z_*=b$. Thus, arguing as in the proof of Lemma \ref{cl}, we  conclude that  $\delta_\pm(t,z) \leq 0$ for all $t \in [0,h]$, $z \in \R$. This estimation shows that actually inequality (\ref{ECO}) is fulfilled for all  $t \in [0,h]$. Now, we can apply step by step procedure in order to  obtain   $\delta_\pm(t,z) \leq 0$ for all $t \geq 0$, $z \in \R$. 

Finally, since $g: [\kappa-\rho, \kappa+\rho]\to [\kappa-\rho, \kappa+\rho]= :\mathcal{I}$ is well defined and 
$$\kappa - \rho \leq m:= \liminf_{z \to +\infty}\phi(z) \leq M:= \limsup_{z \to +\infty}\phi(z) \leq \kappa +\rho, $$
it follows from \cite[Remark 12]{SEDY} that $g([m,M]) \supseteq [m,M]$. On the other hand, $g$ is a contraction on $\mathcal{I}$ so that   $M=m=\kappa$. 
\qed
\end{proof}
\begin{lemma}\label{L3}
Let $g(x)$ and $w_0(t,z)$ meet all the assumptions of Corollary \ref{co3}.    Then inequality (\ref{nev}) implies that the solution $w(t,z)$  of the initial value problem 
$w(s,z) = w_0(s,z), (s,z) \in \Pi_0,$ for (\ref{E1})   satisfies (\ref{ECO}) 
for some positive $\rho, \gamma$. 
\end{lemma}
\begin{proof} Henceforth, we fix small $\epsilon >0$, $\kappa_+ >g(x_m)$ close to $g(x_m)$ and $\kappa_-<g(\kappa_+)$ close to $g(g(x_m))$ such that 
$|g'(x)|<1$ for all 
$x\in [\kappa_-  - \epsilon, \kappa_+  + \epsilon]$.  The latter inequality and the unimodality of $g$ implies that $\kappa$ is a global attractor of the map $g: (0,g(x_m)] \to (0,g(x_m)]$. Therefore  each semi-wavefront $\phi_c$ to equation (\ref{EP}) actually is a wavefront (i.e. $\phi_c(+\infty)=\kappa$, e.g. see \cite[Theorem 18]{SEDY}).  It is easy to see that there exist strictly increasing functions $g_+, g_-: \R_+ \to \R_+$ possessing the following properties: 
\begin{itemize}
\item[(i)] 
$g_-(x)\leq g(x) \leq g_+(x),\ x \in [0,\kappa_+]$;
\item[(ii)] 
$g_-(x)= g(x) = g_+(x)$ for all $x$ from some neighbourhood of $0$;
\item[(iii)] $g_\pm$ satisfies {\rm \bf(H)}  with $\kappa_\pm$ and $L_{g_\pm} = g'(0)$.
\end{itemize}
Let $w_\pm(t,z)$ denote the solution of the initial value problem
$$
w_{t}(t,z)=w_{zz}(t,z)-cw_{z}(t,z)-w(t,z)+g_\pm(w(t-h,z-ch)),$$
$$ w_\pm(s,z)= w_0(s,z), \ (s,z) \in \Pi_0,$$
and let $\phi_\pm$ be wavefront solutions of the stationary equations
$$
0=y''(z)-cy'(z)-y(z)+g_\pm(y(z-ch)).
$$ 
normalised by the condition
$
\lim_{z\to -\infty}\phi_\pm(z)/\phi(z)=1
$
(this is possible in view of (ii)).  Then Theorem \ref{Thm2}A (applied to $w_\pm(t,z)$) and the comparison principle guarantee that there exist  large $b>0$ and $T>h$ such that, for all $t\geq T-h,$ $z \geq b-ch$, it holds  
\begin{equation*}\label{lin}
\kappa_-  - \epsilon < w_-(t,z) \leq w(t,z) \leq w_+(t,z) < \kappa_++\epsilon,\ 
\end{equation*}
\begin{equation*}\label{lina}
\kappa   - \epsilon <   \phi(z) < \kappa+\epsilon. 
\end{equation*}
In addition, by Lemma \ref{L1}, we also can assume that 
\begin{equation*}\label{ECOz}
(\mathcal{L}\delta_\pm)(t,z) \geq 0,\ \delta_\pm(t,z) \leq 0,   \ t \geq T-h,  \ z \leq b, 
\end{equation*}
where $\delta_\pm(t,z)$ are defined by
$$
\delta_-(t,z):= \phi(z) -  (\kappa- \kappa_- +2\epsilon)e^{-\gamma (t-T)}\eta_{\lambda}(z-b) - w(t,z), 
$$
$$
\delta_+(t,z):= w(t,z)-\phi(z) -  (\kappa_+ - \kappa +2\epsilon)e^{-\gamma(t-T)}\eta_{\lambda}(z-b).
$$
Thus  $\delta_\pm(t,z) \leq 0, \ (t,z) \in [T-h,T]\times\R$, so that, arguing as in the proof of Lemma \ref{L7}, we obtain 
$$
(\mathcal{L}\delta_\pm)(t,z) \geq  (|\kappa_\pm - \kappa| +2\epsilon)e^{-\gamma (t-T)}(-m_g+1-\gamma) >0,\   \ t \geq T,  \ z > b, 
$$
together with  $\delta_\pm(t,z) \leq 0$ for all  $t \geq T-h,  z \in \R$. This completes  the proof of Lemma \ref{L3}. \qed
\end{proof}

\begin{acknowledgements} This research was supported by FONDECYT (Chile).  We also thank  Viktor Tkachenko  (Institute of Mathematics in Kyiv, Ukraine) and Robert Hakl (Mathematical Institute in Brno, Czech Republic) for useful discussions:  especially we  would like to acknowledge the support of FONDECYT (Chile), project 1110309 and CONICYT (Chile), project MEC 80130046  which allowed the stay  of Dr.  Tkachenko and Dr. Hakl in the University of Talca. 
\end{acknowledgements}


\begin{thebibliography}{99}

\bibitem{AGT} Aguerrea, M.,  Gomez, C., Trofimchuk, S.:  On uniqueness of semi-wavefronts (Diekmann-Kaper theory of a nonlinear convolution equation re-visited.) {\it  Math. Ann.} {\bf 354}, 73--109 (2012)

\bibitem{AW} Aronson, D.G.,  Weinberger, H.F.:  Nonlinear Diffusion in Population Genetics, Combustion, Nerve Pulse Propagation. {\it Research Notes in Math.} \textbf{14} (London: Pitman), 1--23 (1977)


\bibitem{BD} Benguria, R. D.,  Depassier, M. C.:  Variational characterization
of the speed of propagation of fronts for the nonlinear diffusion
equation. {\it Comm. Math. Phys.} {\bf 175}, 221--227 (1996)

\bibitem{BGHR}  Bonnefon, O., Garnier, J., Hamel, F.,  Roques, L.:  Inside dynamics of delayed traveling waves. {\it Math. Model. Nat. Phenom.}
\textbf{8},  42--59 (2013)

\bibitem{CG1}  Chen, X.,   Guo, J.S.: Existence and asymptotic stability of travelling waves of discrete quasilinear monostable equations. {\it J. Differ. Equa.} {\bf 184}, 549--569 (2002)



\bibitem{CMYZ}  Chern, I.L., Mei, M., Yang, X., Zhang, Q.:
Stability of non-monotone critical traveling waves for
reaction-diffusion equation with time-delay. {\it SIAM J. Math.
Anal.} {\bf 46}, 1053--1084 (2014)


\bibitem{IGT}  Ivanov, A., Gomez, C., Trofimchuk,  S.:
On the existence of non-monotone non-oscillating wavefronts.  {\it
J. Math. Anal. Appl.} \textbf{419}, 606--616  (2014)

\bibitem{FTnon} Faria, T., Trofimchuk, S.:  Positive travelling fronts for reaction-diffusion systems with distributed delay. {\it Nonlinearity} {\bf
23},
2457--2481 (2010)


\bibitem{FML}  Fife, P., McLeod, J.B.: The approach of solutions of nonlinear diffusion equations to travelling front solutions. {\it Arch. Rat. Mech. Anal.} \textbf{65},  335--361 (1977)

\bibitem{AF} Friedman, A. {\it Partial Differential Equations of Parabolic Type}.
Prentice-Hall, Englewood Cliffs, NJ (1964)

\bibitem{JP}  Jankovic, M., Petrovskii, S.:  Are time delays always destabilizing? Revisiting the role of time delays and the Allee effect. {\it Theoretical Ecology}  \textbf{7},
335--349 (2014) 


\bibitem{ID} Garnier, J., Giletti, T., Hamel, F., Roques, L.: Inside dynamics of pulled and pushed fronts. {\it J. de Math\'ematiques Pures et Appliqu\'ees}
\textbf{98},
428--449 (2012)

\bibitem{SEDY}
Gomez, C., Prado, H., Trofimchuk, S.:  Separation dichotomy and wavefronts for a nonlinear convolution equation {\it J. Math. Anal. Appl.} {\bf 420}, 1--19 (2014)


\bibitem{FGT}
Gomez, A., Trofimchuk, S.:  Global continuation of monotone
wavefronts. {\it J.  Lond. Math. Soc. } \textbf{89},  47--68
(2014)


\bibitem{HR} Hadeler, KP., Rothe, F.:
{Travelling fronts in nonlinear diffusion equations}.
 \emph{J. Math. Biol.}     \textbf{2}, 251--263 (1975)

\bibitem{KGB}
 Kyrychko, Y., Gourley, S.A., Bartuccelli, M.V.:  Comparison and convergence to equilibrium in a nonlocal delayed reaction-diffusion model on an infinite domain. {\it Discret Contin. Dyn. Syst. Ser. B.} \textbf{5}, 1015--1026 (2005)

\bibitem{LZh} Liang, X., Zhao, X.-Q.:  Spreading speeds and traveling waves for abstract monostable evolution systems. {\it J. Funct. Anal.} \textbf{259},  857--903
(2010)

\bibitem{LLLM}  Lin, C.K.,  Lin, C.T., Lin, Y., Mei, M.:  Exponential stability of nonmonotone traveling waves for NicholsonÕs blowflies equation. {\it  SIAM J. Math. Anal.} {\bf{46}},
1053--1084 (2014)


\bibitem{LvW} Lv, G., Wang, M.: Nonlinear stability of travelling wave fronts for delayed reaction diffusion equations. {\it  Nonlinearity}  \textbf{23},
845--873 (2010)

\bibitem{ma}
\newblock {Ma, S.:}
\newblock  {Traveling wavefronts for delayed reaction-diffusion systems via a
fixed point theorem}.
\newblock  {\it  J. Differ. Equa}. \textbf{171},
294--314 (2001)


 \bibitem{MaZou} Ma, S., Zou, X.: Existence, uniqueness and stability of traveling waves in a discrete reaction-diffusion
monostable equation with delay. {\it J. Diff. Equa.} \textbf{217}, 54--87 (2005)

\bibitem{morse} {Mallet-Paret, J.:} Morse decompositions for delay-differential
equations. {\em J. Differ. Equa.} {\bf 72 }, 270--315 (1988)

\bibitem{FA} {Mallet-Paret, J.:} The Fredholm alternative for functional
differential equations of mixed type. { \em J. Dynam. Diff. Eqns.}
{\bf 11}, (1999) 1--48.

\bibitem{mps}  {Mallet-Paret, J., and  Sell, G.R.:}  Systems of delay differential
equations I: Floquet multipliers and discrete Lyapunov functions.
{ \em J. Differ. Equa.} {\bf 125},  385--440 (1996)


\bibitem{mps2} {Mallet-Paret, J., and  Sell, G.R.:}   The Poincare-Bendixson
theorem for monotone cyclic feedback systems with delay. { \em J.
Differ. Equa.} {\bf 125}, 441--489 (1996)


\bibitem{MLLS}  Mei, M.,   Lin, C.K.,  Lin C.T.,  So, J.W.-H.: Traveling wavefronts for time-delayed reaction-diffusion equation, (I) Local nonlinearity. {\it  J. Diff. Equa.}  {\bf
247},
495--510 (2009)

\bibitem{MLLS2}  Mei, M.,   Lin, C.K.,  Lin, C.T.,  So, J.W.-H.: Traveling wavefronts for time-delayed reaction-diffusion equation, (II) nonlocal nonlinearity.  {\it J. Diff. Equa.} \textbf{247},
511--529 (2009)

\bibitem{MeiI}
 Mei, M.,  Ou, Ch., Zhao, X.-Q.:
Global stability of monostable traveling waves for nonlocal
time-delayed reaction-diffusion equations. {\it SIAM J. Math.
Anal.}     \textbf{42},  233--258 (2010)

\bibitem{MeiF} Mei, M., So, J.W.-H., Li, M. Y., Shen, S. S. P.: Asymptotic stability of traveling waves for the NicholsonÕs blowflies equation with diffusion. {\it Proc. R. Soc. Edinb. A}
{\bf 134},
579--594 (2004)

\bibitem{MeiW} Mei, M., Wang, Y.:Remark on stability of traveling waves for nonlocal Fisher-KPP
equations. {\it  Intern. J. Num. Anal. Model. Series B}
\textbf{2}, 379--401 (2011)

 \bibitem{OM}  Ogiwara, T., Matano, H.: Monotonicity and convergence in order-preserving systems. {\it Discrete Contin. Dynam. Systems.} \textbf{5}, 1--34
 (1999)



 \bibitem{PW}  Protter, M.H.,  Weinberger, H.F. : {\it Maximum Principles in Differential Equations.}
 Englewood Cliffs, NJ:  Prentice-Hall. (1967)

 \bibitem{RGHK}   Roques, L., Garnier, J., Hamel, F., Klein, E. K.: Allee effect promotes diversity in traveling waves of colonization. {\it Proc. Natl. Acad. Sci. USA}
 {\bf 109},
 8828--8833 (2012)

\bibitem{roth} {Rothe, F.:} {Convergence to pushed fronts.} {\it Rocky Mountain J. Math.} \textbf{11},
617--633 (1981)



\bibitem{ES} van Saarloos, W.: {Front propagation into unstable states.} {\it  Physics Reports. }  
\textbf{386},  29--222 (2003)



\bibitem{STG} Sattinger, D. H.:  On the stability of waves of nonlinear parabolic systems {\it  Adv. Math.} {\bf
22},
312--355 (1976)

\bibitem{sch} { Schaaf, K.:}  Asymptotic behavior
and traveling wave solutions for parabolic functional differential
equations. {\em Trans. Am. Math. Soc.} {\bf 302}, 587--615 (1987)

\bibitem{SZ} Smith, H. L., Zhao, X.-Q.:  Global asymptotic stability of traveling waves in delayed
reaction-diffusion equations. {\it SIAM J. Math. Anal.} \textbf{
31}, 514--534 (2000)

\bibitem{HS} Smith, H. L.:  {\it Monotone Dynamical Systems.
An Introduction to the Theory of Competitive and Cooperative
systems}. AMS, Providence, RI.  (1995)

\bibitem{STR} Solar, A., Trofimchuk, S.:  Asymptotic convergence to a pushed wavefront in monostable  equations with delayed reaction. E-print  arXiv:1408.3344, 27 pp.  (2014)

\bibitem{ST} Stokes, A.N.:
{On two types of moving front in quasilinear diffusion} {\it Math.
Biosciences.}     \textbf{31},  307--315 (1976)

\bibitem{TT}  { Trofimchuk, E., Tkachenko, V., Trofimchuk, S.:} {Slowly
oscillating wave solutions of a single species reaction-diffusion
equation with delay.}  {\it J. Differ. Equa.} \textbf{245},
2307--2332 (2008)


\bibitem{TPT}  Trofimchuk, E., Pinto, M., Trofimchuk, S.:   Pushed traveling fronts in monostable  equations with  monotone delayed reaction.  {\it Discrete Contin. Dyn. Syst.}
\textbf{33},
2169--2187 (2013)


\bibitem{U}   Uchiyama, K.:  The behavior of solutions of some nonlinear diffusion equations for large time. {\it J. Math. Kyoto Univ.} \textbf{18}, 453--508 (1978)

\bibitem{WLR}  Wang,  Z.C., Li, W.T., Ruan, S.:  Travelling fronts in monostable equations with nonlocal delayed effects.  { \em J. Dynam. Diff. Eqns.} \textbf{20}, 563--607 (2008)

\bibitem{wz}
\newblock {Wu J., Zou, X.:},
\newblock {Traveling wave fronts of
reaction-diffusion systems with delay.}
\newblock { \em J. Dynam. Diff. Eqns.} \textbf{13},  651--687 (2001) [Erratum in   { \em J. Dynam. Diff. Eqns.}  {\bf 20}, 531--533 (2008)].

\bibitem{WZL} Wu, S.L., Zhao, H.Q., Liu, S.Y.: Asymptotic stability of traveling waves for delayed reaction-diffusion equations with crossing-monostability. {\it Z. Angew. Math. Phys.} {\bf 62}, 377--397 (2011)

\end{thebibliography}
\end{document}